\documentclass[11pt]{article}\usepackage{epsfig}\usepackage{setspace}%\usepackage{txfonts}
\usepackage{graphicx,psfrag}
 \usepackage{amssymb}\usepackage{amsmath}\usepackage{amsthm}\usepackage{mathptmx}\usepackage[T1]{fontenc}\usepackage{makeidx}
\title{Ancestral branching, cut-and-paste algorithms and associated tree and partition-valued processes}
\author{Harry Crane}
\date{}

\def\sup{\mathop{\rm sup}\nolimits}
\def\per{\mathop{\rm per}\nolimits}

\def\PD{\mathop{\rm PD}\nolimits}

\def\masspartition{\mathop{\mathcal{P}_{\mathrm{m}}}\nolimits}
\def\masspartitionk{\mathop{\mathcal{P}_{\mathrm{m}}^{(k)}}\nolimits}
\def\partitionsn{\mathop{\mathcal{P}_{[n]}}\nolimits}
\def\partitionsnk{\mathop{\mathcal{P}_{[n]}^{(k)}}\nolimits}
\def\partitionsk{\mathop{\mathcal{P}^{(k)}}\nolimits}
\def\anc{\mathop{\mbox{anc}}\nolimits}
\def\root{\mathop{\mbox{root}}\nolimits}
\def\pa{\mathop{\mbox{pa}}\nolimits}
\def\frag{\mathop{\mbox{frag}}\nolimits}
\def\rp{\mathop{\mbox{rp}}\nolimits}
\def\cp{\mathop{\mbox{CP}}\nolimits}

\def\part{\mathop{\mbox{Part}}\nolimits}
\def\subdot{\hbox{\bf.}}
\def\trees{\mathop{\mathcal{T}}\nolimits}
\def\treesn{\mathop{\mathcal{T}_n}\nolimits}
\def\treesk{\mathop{\mathcal{T}^{(k)}}\nolimits}
\def\treesnk{\mathop{\mathcal{T}_n^{(k)}}\nolimits}
\newtheorem{thm}{Theorem}[section]
\newtheorem{lemma}[thm]{Lemma}
\newtheorem{prop}[thm]{Proposition}
\newtheorem{cor}[thm]{Corollary}
\newtheorem{defn}[thm]{Definition}

\begin{document}
\maketitle
%\footnote{University of Chicago, 5734 S. University Ave., Chicago, IL 60637, crane@uchicago.edu}
% "Title of the paper"

%\runtitle{$\varrho_{\nu}$-branching Markov processes}

% indicate corresponding author with \corref{}
% \author{\fnms{John} \snm{Smith}\corref{}\ead[label=e1]{smith@foo.com}\thanksref{t1}}
% \thankstext{t1}{Thanks to somebody} 
% \address{line 1\\ line 2\\ printead{e1}}
% \affiliation{Some University}

%\author{\fnms{Harry} \snm{Crane}\ead[label=e1]{crane@uchicago.edu}}
%\address{University of Chicago\\ Dept.\ of Statistics\\ Eckhart Hall Room 108\\ 5734 S.\ University Avenue\\
%Chicago, IL 60637 U.S.A.}
%\and
%\author{\fnms{???} \snm{???}\ead[label=e2]{???}}
%\address{\printead{e2}}
%\affiliation{???}

%\runauthor{Harry Crane}
\begin{abstract}
We introduce an algorithm for generating a random sequence of fragmentation trees, which we call the ancestral branching algorithm.  This algorithm builds on the recursive partitioning structure of a tree and gives rise to an associated family of Markovian transition kernels whose finite-dimensional transition probabilities can be written in closed-form as the product over partition-valued Markov kernels.  The associated tree-valued Markov process is infinitely exchangeable provided its associated partition-valued kernel is infinitely exchangeable.  We also identify a transition procedure on partitions, called the cut-and-paste algorithm, which corresponds to a previously studied partition-valued Markov process on partitions with a bounded number of blocks.  Specifically, we discuss the corresponding family of tree-valued Markov kernels generated by the combination of both the ancestral branching and cut-and-paste transition probabilities and show results for the equilibrium measure of this process, as well as its associated mass fragmentation-valued and weighted tree-valued processes.
\end{abstract}

\section{Some preliminaries}\label{section:preliminaries}
A {\em set partition} $B$ of the natural numbers $\mathbb{N}$ is a collection $\{B_1,B_2,\ldots\}$ of disjoint non-empty subsets of $\mathbb{N}$, called {\em blocks}, such that $\bigcup_iB_i=\mathbb{N}$.  In general, we assume the blocks of $B$ are unordered, but whenever we wish to emphasize that blocks are listed in a particular order we write $B=(B_1,B_2,\ldots)$.  Write $\mathcal{P}$ to denote the space of set partitions of $\mathbb{N}$.

For $B\in\mathcal{P}$ and $b\in B$, $\#B$ is the number of blocks of $B$ and $\#b$ is the number of elements of $b$.  We write $\partitionsk$ to denote the space of partitions of $\mathbb{N}$ with at most $k\geq1$ blocks, i.e.\ $\partitionsk:=\{B\in\mathcal{P}:\#B\leq k\}$.  For a partition $B$ with blocks $\{B_1,B_2,\ldots\}$ and any $A\subset\mathbb{N}$, let $B_{|A}$ denote the {\em restriction} of $B$ to $A$, i.e.\ $B_{|A}:=\{B_i\cap A: i\geq1\}$ (excluding the empty set).  We write $\mathcal{P}_A$ and $\mathcal{P}_A^{(k)}$ to denote the restriction to $A$ of $\mathcal{P}$ and $\partitionsk$ respectively.  In particular, for $n\in\mathbb{N}$, $\partitionsn$ and $\partitionsnk$ are the restriction to $[n]:=\{1,\ldots,n\}$ of $\mathcal{P}$ and $\mathcal{P}^{(k)}$ respectively.

For each $n\in\mathbb{N}$, we define the {\em deletion} operation $D_n:2^{\mathbb{N}}\rightarrow2^{\mathbb{N}}$ which acts on subsets of $\mathbb{N}$ by removing $\{n\}$ from $A$, i.e.\ $A\mapsto D_nA:=A\backslash\{n\}$ for each $A\subset\mathbb{N}$.  In general, for $A,B\subset\mathbb{N}$ non-empty, $D_BA:=A\backslash B=A-B=A\cap B^c.$  For each $n\geq1$, we define the deletion operation on partitions $D_{n,n+1}:\mathcal{P}_{[n+1]}\rightarrow\partitionsn$ in terms of $D_{n+1}$ by $D_{n,n+1}B\equiv B_{|[n]}:=\{D_{n+1}b:b\in B\}$ for every $B\in\mathcal{P}_{[n+1]}$, and for $m<n$ define $D_{m,n}:=D_{m,m+1}\circ\cdots\circ D_{n-1,n}$.  The finite spaces $(\partitionsn,n\geq1)$ together with all deletion $(D_{m,n},m\leq n)$ and permutation maps, and their compositions, defines a {\em projective system} of set partitions.

A sequence $(B_1,\ldots)$ such that $B_n\in\partitionsn$ for each $n\geq1$ is said to be {\em compatible} if $B_n=D_{n,n+1}B_{n+1}$ for each $n\geq1$.  Any $B\in\mathcal{P}$ can be represented as the compatible sequence of its finite restrictions, $(B_{|[n]},n\geq1)$, and we often write $B:=(B_{|[n]},n\geq1)$.

\subsection{Fragmentation trees}
For any subset $A\subset\mathbb{N}$, a collection of non-empty subsets $T\subset2^A$, the power set of $A$, is an $A$-labeled {\em rooted tree} if
\begin{itemize}
	\item[(i)] $A\in T$, called the {\em root} of $T$ and denoted $\root(T)=A$, and
	\item[(ii)] $A,B\in T$ implies $A\cap B\in\{\emptyset,A,B\}$.  That is, either $A$ and $B$ are disjoint or one is a subset of the other.
\end{itemize}
If $T$ contains all singleton subsets of $A$, $T$ is called a {\em fragmentation tree}.  Throughout the rest of this paper, the word {\em tree} and {\em fragmentation} are both understood to mean fragmentation tree.  We write $\mathcal{T}_A$ to denote the space of fragmentations of $A$ and $\mathcal{T}\equiv\mathcal{T}_{\mathbb{N}}$ to denote the space of fragmentations of $\mathbb{N}$.

As a collection of subsets of $A\subset\mathbb{N}$, the elements of $T\in\mathcal{T}_A$ are partially ordered by {\em inclusion}.  That is, if $A,B\in T$ such that $A\subset B$, then the intervals $[A,B],(A,B],$ and $[A,B)$ are well-defined subsets of $T$.  This partial ordering induces a natural genealogical interpretation of the relationships among the elements of a tree.  For each $t\in T$, the subset $\anc(t):=(t,A]:=\{s\in T:t\subset s\}$ denotes the set of {\em ancestors} of $t$.  Note that $\anc(\root(T))=\emptyset$ and for each $t\neq\root(T)$, $\anc(t)$ has a least element denoted by $\pa(t):=\min\anc(t)$, the {\em parent} of $t$.

Conversely, except for the singleton elements of $T$, each $t\in T$ is the parent of some collection of subsets of $T$, called the {\em children} of $t$, which is given by $\pa^{-1}(t):=\frag(t):=\{t'\in T:\pa(t')=t\}$.  For finite $A\subset\mathbb{N}$ and $T\in\mathcal{T}_A$, $\frag(t)$ forms a non-trivial partition of $t$ for each non-singleton $t\in T$.  In particular, for each finite subset $A\subset\mathbb{N}$ and any tree $T\in\mathcal{T}_A$, the children of $\root(T)$ form a well-defined {\em root partition}, denoted $\Pi_T:=\rp(T):=\frag(\root(T)).$  The {\em fragmentation degree} of $T$ is given by $\max_{t\in T}\#\frag(t)$, which may be infinite.  For $k\geq1$, we write $\mathcal{T}_A^{(k)}$ to denote the collection of trees of $A$ with fragmentation degree at most $k$.

For any subset $S\subset A$, the {\em restriction} of $T\in\mathcal{T}_A$ to $S$ is defined by $T_{|S}:=\{S\cap t:t\in T\}$ (excluding the empty set), the {\em reduced sub-tree} of Aldous \cite{AldousCRTI}.  Recall the deletion operation $D_S:2^{\mathbb{N}}\rightarrow2^{\mathbb{N}}$ defined above by {\em restriction to the complement of} $S$.  For any tree $T\in\mathcal{T}_A$ and $S\subset A$, $D_ST:=\{D_St:t\in T\}=\{t\cap S^c:t\in T\}\equiv T_{|A\cap S^c}$.  We use the notation $D_{n,n+1}:\mathcal{T}_{n+1}\rightarrow\mathcal{T}_n$ to denote the operation $D_{n,n+1}T:=T_{|[n]}$ on trees.  Note that the apparent overloading of $D_{n,n+1}$ as a function on both $\mathcal{P}_{[n+1]}$ and $\mathcal{T}_{n+1}$ should cause no confusion as it is fundamentally defined, in both cases, as a function on collections of subsets of $\mathbb{N}$ through the set operation $D_{n+1}$.

As in the description of partitions of $\mathbb{N}$, any fragmentation $T\in\mathcal{T}$ can be expressed as a compatible sequence $(T_{|[n]},n\geq1)$ of reduced subtrees on the projective system of $[n]$-labeled trees $(\mathcal{T}_n,n\geq1)$ together with deletion $(D_{m,n},m\leq n)$ and permutation maps.  For $T\in\mathcal{T}$, we often write $T:=(T_{|[n]},n\geq1)$.
\section{Summary of main results}\label{section:summary of main results}
Our main result is the description of an explicit random algorithm for generating a sequence of fragmentation trees and conditions under which this algorithm characterizes an infinitely exchangeable Markov process on $\mathcal{T}$, which turns out to be quite general.  Later, we discuss a special subclass of this family of tree-valued processes for which we can establish the Feller property and existence of associated processes on mass fragmentations and weighted trees.  This subfamily can give rise to an infinitely exchangeable process on, for example, binary trees, which could have implications in certain areas of inference for unknown phylogenetic trees.  The associated weighted tree-valued process may also be applicable to certain aspects of hidden Markov modeling in a phylogenetic setting.

Previously, random algorithms, e.g.\ subtree prune-regraft (SPR), genetic algorithms, neighbor-joining, etc., have been described in the context of Markov chain Monte Carlo (MCMC) and searching the space of trees in the context of inference of unknown phylogenetic trees, see e.g.\ Felsenstein \cite{Felsenstein2004} for an overview.    In particular, Evans and Winter \cite{EvansWinter2006} study a tree-valued process based on an SPR algorithm which is reversible with respect to Aldous's continuum random tree (CRT) \cite{AldousCRTI}.  Previously, Aldous and Pitman \cite{AldousPitmanTree} studied a tree-valued process based on SPR and its connection to the Galton-Watson process.

 Below we introduce a random algorithm, which we call the {\em ancestral branching} algorithm, which is of a different nature than those previously studied in this context and generates different sample paths on the space of fragmentation trees than its predecessors.  This procedure admits an explicit expression for finite-dimensional Markovian transition probabilities which is of an intuitive form, and can be related to the notion of successive partitioning of a set which is common in the study of fragmentation processes.  We subsequently show a construction of an infinitely exchangeable process which evolves according to ancestral branching, as well as connections to Poisson point processes, mass fragmentations and weighted trees.
\subsection{Ancestral branching kernels}\label{section:ancestral branching algorithm}
A {\em Markov kernel} on a set $\mathcal{A}$ is a collection $\{p(x,\cdot):x\in\mathcal{A}\}$ of probability distributions on $\mathcal{A}$ indexed by the elements of $\mathcal{A}$.  In particular, for any $A\subset\mathbb{N}$, a Markov kernel on $\mathcal{P}_A$ is a collection $P_A:=\{p(B,\cdot):B\in\mathcal{P}_A\}$ of probability distributions on $\mathcal{P}_A$ indexed by the elements of $\mathcal{P}_A$.

Let $A\subset\mathbb{N}$ be a finite subset such that $\#A\geq2$ and let $\{P_S:S\subseteq A\}$ be a collection of Markov kernels on $\mathcal{P}_S$ for all $S\subseteq A$. Given $T\in\mathcal{T}_A$, a fragmentation of $A$, generate a new fragmentation $T'\in\mathcal{T}_A$ by the following procedure.
\paragraph{Ancestral Branching (AB) Algorithm}

\begin{enumerate}
	\item[(i)] Put $F:=\{A\}$.
	\item[(ii)] Pick any $b$ from $F$ such that $\#b\geq2$.
	\item[(iii)] Generate $\pi_b$ from $p_b(\Pi_{T_{|b}},\cdot)$, the transition measure on $\mathcal{P}_b$ with initial state given by the root partition of the reduced subtree $T_{|b}$, independently of everything generated previously.  
	\item[(iv)] If $\pi_b={\bf1}_b\equiv\{b\}$, discard and repeat step (iii) for $b$; otherwise, put $\Pi_{T'_{|b}}=\pi_b$, i.e.\ define the children of $b$ in $T'$, $\frag(b)$, by the blocks of $\pi_b$.
	\item[(v)] Remove $b$ from $F$ and add the blocks of $\pi_b$ to F, i.e.\ $F\mapsto (F-b)\cup\frag(b)$.
	\item[(vi)] If there is a non-singleton element of $F$, i.e.\ $\#\{b\in F:\#b\geq2\}>0$, go to (ii); otherwise, stop.
\end{enumerate}
If we assume for each $b\subseteq A$ that $p_b(B,{\bf1}_b)<1$ for each $B\in\mathcal{P}_b$, then $\frag(b)$ is almost surely generated in a finite number of steps in (iii) and (iv).  By assuming $A$ is a finite set, we have that the above algorithm runs in a finite number of steps with probability one.  

Henceforth, we shall assume the partition-valued kernels $\{p_b(\cdot,\cdot):b\subset\mathbb{N}\}$ satisfy $p_b(\cdot,{\bf1}_b)<1$ for every $b\subset\mathbb{N}$.  Under this condition, it is straightforward to show that the above algorithm culminates in a transition probability $Q_A(T,\cdot)$ on $\mathcal{T}_A$, which we can express in closed form by
\begin{equation}Q_A(T,T')=\prod_{b\in T':\#b\geq2}\frac{p_b(\Pi_{T_{|b}},\Pi_{T'_{|b}})}{1-p_b(\Pi_{T_{|b}},{\bf1}_{b})},\label{eq:branching Markov kernel}\end{equation}
the product of Markov kernels on the root partitions of the reduced subtrees of all parents of $T'$ conditioned to be non-trivial, i.e.\ not the one block partition ${\bf1}_b$.

To see this, note that for each $b\in T'$ we generate $\Pi_{T'_{|b}}$ independently of all other random partitions generated by this algorithm.  Therefore, we can write $Q_A$ as a product over $\{b\in T':\#b\geq2\}$ of conditional probabilities $\mathbb{P}_b(\Pi_{t'_{|b}}=\pi|T=t)$, i.e.\
 $$Q_A(t,t')=\prod_{b\in t':\#b\geq2}P_b(\Pi_{t'_{|b}}=\pi|T=t).$$
From (iii) and (iv), we have that 
$$P_b(\Pi_{T'_{|b}}=\pi_b|T=t)=\sum_{i=0}^{\infty}p_b(\Pi_{t_{|b}},\pi_b)p_b(\Pi_{t_{|b}},{\bf1}_b)^i=\frac{p_b(\Pi_{t_{|b}},\pi_b)}{1-p_b(\Pi_{t_{|b}},{\bf1}_b)}$$
for each $b\in t'$, which gives us \eqref{eq:branching Markov kernel}.  By a straightforward induction argument, one can easily show that the sum of \eqref{eq:branching Markov kernel} over the elements of $\mathcal{T}_A$ equals one, and so \eqref{eq:branching Markov kernel} defines a Markov kernel on $\mathcal{T}_A$.

We call any Markov kernel on $\mathcal{T}_A$ of the form \eqref{eq:branching Markov kernel} an {\em ancestral branching (AB)} Markov kernel on $\mathcal{T}_A$.  It is clear that transitions $T\mapsto T'$ on $\mathcal{T}_A$ governed by an AB kernel $Q_A(\cdot,\cdot)$ can be generated according to the AB algorithm by taking the transition probabilities in step (iii) to be the $p_b(\cdot,\cdot)$ used in the product of \eqref{eq:branching Markov kernel}.

For $A\subset\mathbb{N}$ with $2\leq\#A<\infty$, the form of \eqref{eq:branching Markov kernel} admits the recursive expression
\begin{equation} Q_A(T,T')=\frac{p_A(\Pi_T,\Pi_{T'})}{1-p_A(\Pi_T,{\bf1}_A)}\prod_{b\in \Pi_{T'}}Q_b(T_{|b},T'_{|b}),\label{eq:recursive branching kernel}\end{equation}
which has an intuitive interpretation in terms of independent self-similar transitions on the space of reduced subtrees of the children of the root of $T'$.  The reader familiar with the literature on fragmentation processes may draw parallels to the usual description of a fragmentation process in terms of successive partitioning of fragments, see e.g.\ \cite{Bertoin2006, McCullaghPitmanWinkel2008}.  Indeed, the specification in \eqref{eq:branching Markov kernel} is related to this specification, but has the added feature of including a Markovian dependence on the previous state in a sequence of fragmentation trees, which has not previously appeared in the study of tree-valued processes.

The Markovian branching algorithm in section \ref{section:ancestral branching algorithm} only requires associated $\mathcal{P}$-valued transition probabilities to be defined on $\mathcal{P}_S\backslash\{{\bf1}_S\}$ for each $S\subseteq A$.  However, in our treatment we always assume that we have a family of transition probabilities which is well-defined on the full space $\mathcal{P}_S$ and satisfies $p_S(\cdot,{\bf1}_S)<1$.  This distinction becomes necessary when we consider infinitely exchangeable processes of AB type later on.

The rest of this paper is organized as follows.  In section \ref{section:infinitely exchangeable processes}, we discuss general conditions under which the AB algorithm gives rise to an infinitely exchangeable tree-valued process.  Section \ref{section:cut-and-paste algorithm} introduces an algorithm on set partitions, the {\em cut-and-paste} (CP) algorithm, and draws parallels to an infinitely exchangeable partition-valued process in \cite{Crane2011a}.  Section \ref{section:cut-and-paste ancestral branching} shows some special properties of the associated tree-valued process based on the combination of both the AB and CP algorithms.
\section{Infinitely exchangeable processes}\label{section:infinitely exchangeable processes}
Infinitely exchangeable random partitions and partition-valued processes have been studied in some detail in the literature.  Ewens \cite{Ewens1972} first introduced his sampling formula as a model in population genetics, which was later studied as a process on set partitions by Kingman \cite{Kingman1978} and several others.  Coalescent processes \cite{EvansPitman1998,Kingman1980,Kingman1982}, fragmentation processes \cite{Bertoin2001a,Bertoin2002,Bertoin2003a,McCullaghPitmanWinkel2008}, fragmentation-coalescence processes \cite{Berestycki2004, DurrettGranovskyGueron1999}, and other general processes \cite{Crane2011a} are partition-valued processes for which conditions for infinite exchangeability have been discovered.  Given the form of the finite-dimensional transition probabilities in \eqref{eq:branching Markov kernel} and its apparent relationship to partition-valued processes, we study conditions under which this tree-valued process is infinitely exchangeable.
\subsection{Exchangeable ancestral branching Markov kernels}
A collection of Markov kernels $Q:=\{Q_A(\cdot,\cdot):A\subseteq\mathbb{N}\}$ on $(\mathcal{T}_A,A\subset\mathbb{N})$ is {\em finitely exchangeable} if for each $n\geq1$, $A,B\subset\mathbb{N}$ with $\#A=\#B=n$, and $t\in\mathcal{T}_A$
\begin{equation}Q_A(t,\cdot)=Q_{B}(\varphi^*(t),\varphi^*(\cdot))\label{eq:exchangeable Markov kernel}\end{equation}
for every one-to-one injection map $\varphi:A\rightarrow B$, where $\varphi^*:\mathcal{T}_A\rightarrow\mathcal{T}_B$ is its associated injection $\mathcal{T}_A\rightarrow\mathcal{T}_B$.  In other words, $Q_B\equiv Q_A\varphi^{*-1}$, the distribution induced on $\mathcal{T}_B$ by $Q_B$ and the injection $\varphi$.  In this case, there exists a map $\sigma_A:A\rightarrow[n]$ such that $Q_A(\cdot,\cdot)=Q_n(\sigma_A(\cdot),\sigma_A(\cdot))=:Q_n\sigma_A(\cdot,\cdot)$, the {\em exchangeable transition probability function} for $n$.  

We define the {\em canonical injection} $A\rightarrow[n]$ as follows.  Suppose, without loss of generality, that $A=\{a_1,\ldots,a_n\}$ with $a_1<a_2<\ldots<a_n$.  Then we define the canonical injection by $\varphi_A:A\rightarrow[n],a_i\mapsto i$.  For each $A\subseteq\mathbb{N}$ such that $\#A=n$, we have $Q_A(\cdot,\cdot)=Q_n\varphi_A(\cdot,\cdot)$.  Therefore, for a finitely exchangeable family of Markovian transition probabilities, we need only specify a transition probability $Q_n(\cdot,\cdot)$ on $\treesn$ for each $n\geq1$.

\begin{thm}\label{thm:exchangeable branching kernel}Let $n\geq1$ and for each $A\subset\mathbb{N}$ with $\#A=n$ let $Q_A(\cdot,\cdot)$ be a branching Markov kernel on $\mathcal{T}_A$ defined by the family $\{P_S^*:S\subseteq A\}$, where $P_S^*:=\{p_S(B,\cdot):B\in\mathcal{P}^*_S:=\mathcal{P}_S\backslash\{{\bf1}_S\}\}$.  Assume further that for every finite $A,B\subset\mathbb{N}$ with $\#A=\#B$ and injection $\psi^*:\mathcal{P}_A\rightarrow\mathcal{P}_B$, $p_{A}(\pi,{\bf1}_A)=p_B(\psi^*(\pi),{\bf1}_B)$.  Then the family $\{Q_A:A\subset\mathbb{N}\}$ is finitely exchangeable if and only if the restricted collection $\{P_S^*:S\subset \mathbb{N}\}$ is finitely exchangeable.\end{thm}
\begin{proof}
Let $A\subset\mathbb{N}$ be a finite subset and $P:=\{P_S:S\subseteq A\}$ be some family of Markov kernels on $\{\mathcal{P}_S:S\subseteq A\}$.  From \eqref{eq:branching Markov kernel}, the AB Markov kernel on $\mathcal{T}_A$ based on $P$ is
$$Q_A(T,T')=\prod_{b\in T':\#b\geq2}\frac{p_b(\Pi_{T_{|b}},\Pi_{T'_{|b}})}{1-p_b(\Pi_{T_{|b}},{\bf1}_b)}.$$

For $A,B\subset\mathbb{N}$ with $\#A=\#B$ and injection map $\varphi:A\rightarrow B$ with associated injection $\varphi^{*}:\mathcal{T}_A\rightarrow\mathcal{T}_B$, we also write $\varphi^*$ to denote the associated injection $\mathcal{P}_A\rightarrow\mathcal{P}_B$, which should cause no confusion since it is clear from context to which we are referring.

For $n=2$ and $A,B\subset\mathbb{N}$ such that $\#A=\#B=2$ with injection map $\varphi:A\rightarrow B$ and associated injection $\psi^*:\mathcal{P}_A\rightarrow\mathcal{P}_B$.  In this case, $\#\mathcal{P}_A=\#\mathcal{P}_B=2$ so that we can write $A_1,A_2$ as the elements of $\mathcal{P}_A$ with $\#A_1=1$ and $\#A_2=2$.  Likewise, we write $B_1$ and $B_2$ for one and two block (respectively) elements of $\mathcal{P}_B$.  Hence, $\psi^*(A_i)=B_i$ for $i=1,2$.  It is assumed that $p_B(\psi^*(\pi),{\bf1}_B)=p_A(\pi,{\bf1}_A)$ for each $\pi\in\mathcal{P}_A$.  Hence, $p_B(\psi^*(\pi),{\bf1}_B)=p_B(\psi^*(\pi),B_1)=p_A(\pi,A_1)$ and $1-p_B(\psi^*(\pi),B_1)=p_B(\psi^*(\pi),B_2)=p_A(\pi,A_2)=1-p_A(\pi,A_1)$ and $p_B=p_A\psi^{*-1}$ for $\#A=\#B=2$.  So $\{p_A(\cdot,\cdot):\#A=2\}$ is exchangeable.  Also, $\#\mathcal{T}_A=\#\mathcal{T}_B=1$ implies for $t\in\mathcal{T}_A$, $Q_A(t,t)=Q_B(\psi^*(t),\psi^*(t))=1$ trivially.  So we have that $\{Q_A(\cdot,\cdot):\#A=2\}$ is exchangeable.

Now, fix $n>2$ and suppose that for any pair $A,B\subset\mathbb{N}$ with $\#A=\#B\leq n$ and any injective map $\varphi:A\rightarrow B$ we have that $Q_B=Q_A\varphi^{*-1}$ implies that $p_B=p_A\varphi^{*-1}$ on $\mathcal{P}_B^*.$  Now, consider $A^*,B^*\subset\mathbb{N}$ with $\#A^*=\#B^*=n+1$ and let $\psi:A^*\rightarrow B^*$ be the unique injective map $A^*\rightarrow B^*$ whose restriction to $A\rightarrow B$ corresponds to $\varphi$. Write $\psi^*:\mathcal{T}_{A^*}\rightarrow\mathcal{T}_{B^*}$ for its associated injection $\mathcal{T}_{A^*}\rightarrow\mathcal{T}_{B^*}$.  

Assume that $Q_{A^*}=Q_{B^*}\psi^{*}$ and let $t,t'\in\mathcal{T}_{A^*}$.  We have the following.
\begin{eqnarray}
Q_{A^*}(t,t')&=&\frac{p_{A^*}(\Pi_t,\Pi_{t'})}{1-p_{A^*}(\Pi_t,{\bf1}_{A^*})}\prod_{b\in\Pi_{t'}:\#b\geq2}Q_b(t_{|b},t'_{|b})\\
&=&\frac{p_{B^*}(\Pi_{\psi^*(t)},\Pi_{\psi^*(t')})}{1-p_{B^*}(\Pi_{\psi^*(t)},{\bf1}_{B^*})}\prod_{b\in\Pi_{\psi^*(t')}:\#b\geq2}Q_{b}(\psi^*(t)_{|b},\psi^*(t')_{|b})\\
&=&\frac{p_{B^*}(\Pi_{\psi^*(t)},\Pi_{\psi^*(t')})}{1-p_{B^*}(\Pi_{\psi^*(t)},{\bf1}_{B^*})}\prod_{b\in\Pi_{\psi^*(t')}:\#b\geq2}Q_{b}(\psi^*(t_{|\psi^{-1}(b)}),\psi^*(t'_{|\psi^{-1}(b)}))\\
&=&\frac{p_{B^*}(\Pi_{\psi^*(t)},\Pi_{\psi^*(t')})}{1-p_{B^*}(\Pi_{\psi^*(t)},{\bf1}_{B^*})}\prod_{b\in\Pi_t':\#b\geq2}Q_b(t_{|b},t'_{|b})
\end{eqnarray}
which implies that 
$$\frac{p_{A^*}(\pi,\pi')}{1-p_{A^*}(\pi,{\bf1}_{A^*})}=\frac{p_{B^*}(\psi^*(\pi),\psi^*(\pi'))}{1-p_{B^*}(\psi^*(\pi),{\bf1}_{B^*})}$$ for all one-to-one functions $\psi^*:\mathcal{P}_{A^*}\rightarrow \mathcal{P}_{B^*}$ and all $\pi,\pi'\in\mathcal{P}_{A^*}\backslash\{{\bf1}_{A^*}\}=:\mathcal{P}_{A^*}^*$, which establishes that
\begin{eqnarray*}
p_{A^*}(\pi,\pi')&=&p_{B^*}(\psi^*(\pi),\psi^*(\pi'))
\end{eqnarray*}
by assumption that $p_{B^*}(\psi^*(\cdot),{\bf1}_{B^*})=p_{A^*}(\cdot,{\bf1}_{A^*})$ for all $A^*,B^*$ such that $\#A^*=\#B^*$ and any injective mapping $\psi:A^*\rightarrow B^*$.

This establishes finite exchangeability for $\{p_A(\cdot,\cdot):\#A\leq n+1\}$.  Induction implies this holds for all $n\geq1$ and hence also implies finite exchangeability of $\{p_A(\cdot,\cdot):A\subset\mathbb{N},\#A<\infty\}$.

The reverse implication is obvious.  In fact, if $\{P_S^*:S\subset\mathbb{N}\}$ is finitely exchangeable, then $p_A(\pi,{\bf1}_A)=p_B(\psi^*(\pi),{\bf1}_B)$ for any $A,B$ with $\#A=\#B$ and any injection $\psi^*:\mathcal{P}_{A}\rightarrow\mathcal{P}_B$.  So the additional assumption in the statement of the theorem is implicit.
\end{proof}
Theorem \ref{thm:exchangeable branching kernel} establishes a correspondence between collections of exchangeable Markov kernels on $\mathcal{P}_{[n]}$ such that $p_n(B,{\bf1}_n)<1$ for each $n\geq1$ and exchangeable ancestral branching Markov kernels on $\mathcal{T}_n$.  For all practical purposes, it is sufficient to have an exchangeable Markov process on $\partitionsn$.  There are several known results for exchangeable processes on the projective system $(\partitionsn,n\geq1)$, e.g.\ exchangeable coagulation-fragmentation (EFC) process \cite{Berestycki2004}, the $\cp(\nu)$-Markov process \cite{Crane2011a} which we shall call the cut-and-paste process in light of the exposition in section \ref{section:cut-and-paste algorithm}, and any properties of the induced $\mathcal{T}$-valued process associated with either of these are of interest.  As we see in section \ref{section:cut-and-paste ancestral branching}, the the cut-and-paste ancestral branching process lends itself to certain extensions.
\subsection{Consistent ancestral branching kernels}\label{section:consistent branching kernels}
Let $A\subseteq\mathbb{N}$.  A family of Markov kernels $\{Q_S:S\subseteq A\}$ defined on the projective system $\{\mathcal{T}_S:S\subseteq A\}$ is {\em consistent} if for all $\emptyset\neq C\subset B\subseteq A$, $t\in\mathcal{T}_C$ and $t^*\in D^{-1}_{C,B}(t)$,
\begin{equation}Q_BD^{-1}_{C,B}(t,\cdot):=Q_B(t^*,D^{-1}_{C,B}(\cdot))=Q_C(t,\cdot).\label{eq:consistent Markov kernel}\end{equation}

In other words, for any $C\subset B$ and injection $\varphi:C\rightarrow B$ with associated projection $\varphi^*:\mathcal{T}_B\rightarrow\mathcal{T}_C$, we have $Q_C\equiv Q_B\varphi^{*-1}$.
\begin{thm}\label{thm:consistent branching kernel}Let $Q:=\{Q_S:S\subseteq A\}$ be a family of ancestral branching Markov kernels based on a collection $P:=\{P_S:S\subseteq A\}$.  The family $Q$ is consistent if each $p_S(\pi,\cdot)$ is consistent for all $\pi$ such that $\pi\neq{\bf1}_S$.

Moreover, if, in addition, $p_{S^*}(\pi^*,{\bf e}_{S^*})+p_{S^{*}}(\pi^*,{\bf1}_{S^*})=p_S(\pi,{\bf1}_S)$ for every $S\subset S^*$ with $\#S=\#S^*-1$ and every $\pi\in\mathcal{P}_S$ and $\pi^*\in D_{S,S^*}^{-1}(\pi)$, then $Q$ consistent implies $p_S(\cdot,\cdot)$ is consistent for all $S\subseteq A$.
\end{thm}
\begin{proof}
For $S\subseteq A$ and $x\in S\cap A^{\mbox{c}}$, write $S^x:=S\cup\{x\}$.

Suppose $Q$ is consistent and $p_{S^x}(\pi^*,{\bf e}_{S^x})+p_{S^{x}}(\pi^*,{\bf1}_{S^x})=p_S(\pi,{\bf1}_S)$ for every $\pi\in\mathcal{P}_S$ and $\pi^*\in D^{-1}_{S,S^x}(\pi)$.  Then we show that $P$ is consistent by induction.  For $S\subset A$ such that $\#S=2$, we have that $\mathcal{T}_S$ contains exactly one element, which we denote $t_S$.  It is clear that $Q_S(t_S,t_S)=1$ and for any $S^x\subseteq A$ we have that 
$$\sum_{t''\in D^{-1}_{S,S^x}(t_S)}Q_{S^x}(t^*,t'')=\sum_{t''\in\mathcal{T}_{S^x}}Q_{S^x}(t^*,t'')=\sum_{\pi\in\mathcal{P}_{S^x}\backslash\{{\bf1}_{S^x}\}}\frac{p_{S^x}(\Pi_{t^*},\pi)}{q_{S^x}(\Pi_{t^*},{\bf1}_{S^x})}=1$$
for any $t^*\in\mathcal{T}_{S^x}$ by the fact that $Q_{S^x}$ is a transition probability.  By our assumption, we have
\begin{eqnarray*}
1-p_{S^x}(\Pi_{t^*},{\bf1}_{S^x})&=&\sum_{\pi\in D^{-1}_{S,S^x}(\Pi_{t'})}p_{S^x}(\Pi_{t^*},\pi)+[p_S(\Pi_t,{\bf1}_S)-p_{S^x}(\Pi_{t^*},{\bf1}_{S^x})]\\
p_S(\Pi_t,\Pi_{t'})+p_S(\Pi_t,{\bf1}_S)&=&\sum_{\pi\in D^{-1}_{S,S^x}(\Pi_{t'})}p_{S^x}(\Pi_{t^*},\pi)+p_S(\Pi_t,{\bf1}_S)\\
p_S(\Pi_t,\Pi_{t'})&=&\sum_{\pi\in D^{-1}_{S,S^x}(\Pi_{t'})}p_{S^x}(\Pi_{t^*},\pi),
\end{eqnarray*}
and $p_{S}$ is consistent with $p_{S^x}$ for $\#S=2$ and $S^x$.

Now, for each $S\subset A$ with $\#S=m<\#A$, assume that $p_T(\cdot,\cdot)$ is consistent for all $T\subseteq S$, and let $S^x=S\cup\{x\}$ for some $x\in A\cap S^{\mbox{c}}$.  Assume $t,t'\in\mathcal{T}_S$ and let $t^*\in D^{-1}_{S,S^*}(t)$.  For a partition $\pi\in\mathcal{P}_S$ and $b\in\pi$, write $b^x\in\pi^*\in D^{-1}_{S,S^x}(\pi)$ to denote the block of $\pi$ to which $x$ is added to obtain $\pi^*$.  We have
\begin{eqnarray}
\lefteqn{\sum_{t''\in D^{-1}_{S,S^*}(t')}Q_{S^*}(t^*,t'')=}\\
&=&\sum_{\pi^*\in D^{-1}_{S,S^x}(\Pi_{t'})}\sum_{t''\in D^{-1}_{b,b^x}(t'_{|b})}\frac{p_{S^x}(\Pi_{t^*},\pi^*)}{q_{S^x}(\Pi_{t^*},{\bf1}_{S^x})}\prod_{b\in\pi^*:b\neq b^x}\left[Q_b(t^*_{|b},t'_{|b})\right]Q_{b^x}(t^*_{|b^x},t''_{|b^x})+\frac{p_{S^x}(\Pi_{t^*},{\bf e}_{S^x})}{q_{S^x}(\Pi_{t^*},{\bf1}_{S^x})}Q_S(t,t')\label{consistency line 1}\\
&=&\sum_{\pi^*\in D^{-1}_{S,S^x}(\Pi_{t'})}\frac{p_{S^x}(\Pi_{t^*},\pi^*)}{q_{S^x}(\Pi_{t^*},{\bf1}_{S^x})}\prod_{b\in\pi^*:b\neq b^x}Q_b(t^*_{|b},t''_{|b})\sum_{t''\in D^{-1}_{b,b^x}(t'_{|b})}Q_{b^x}(t^*_{|b^x},t''_{|b^x})+\frac{p_{S^x}(\Pi_{t^*},{\bf e}_{S^x})}{q_{S^x}(\Pi_{t^*},{\bf1}_{S^x})}Q_S(t,t')\label{consistency line 2}\\
&=&\sum_{\pi^*\in D^{-1}_{S,S^x}(\Pi_{t'})}\frac{p_{S^x}(\Pi_{t^*},\pi^*)}{q_{S^x}(\Pi_{t^*},{\bf1}_{S^x})}\prod_{b\in \pi}Q_b(t_{|b},t'_{|b})+\frac{p_{S^x}(\Pi_{t^*},{\bf e}_{S^x})}{q_{S^x}(\Pi_{t^*},{\bf1}_{S^x})}Q_S(t,t')\label{consistency line 3}\\
&=&Q_{S}(t,t')\left[\frac{p_{S^x}(\Pi_{t^*},{\bf e}_{S^x})}{q_{S^x}(\Pi_{t^*},{\bf1}_{S^{x}})}+\sum_{\pi^*\in D^{-1}_{S,S^*}(\Pi_{t'})}\frac{p_{S^x}(\Pi_{t^*},\pi^*)}{q_{S^x}(\Pi_{t^*},{\bf1}_{S^x})}\frac{q_{S}(\Pi_t,{\bf1}_{S})}{p_S(\Pi_t,\Pi_{t'})}\right].\label{consistency line 4}
\end{eqnarray}
Here, \eqref{consistency line 1} follows by noticing that the restriction $t^*_{|b}$ and $t'_{|b}$ is unaffected unless $b=b^x$ and that $t''\in D^{-1}_{S,S^*}(t')$ can be broken down into a sum over $\pi^*\in D^{-1}_{S,S^x}(\Pi_{t'})$ and a sum over trees in the inverse image of the reduced subtree $t'_{|b^x}$.  Line \eqref{consistency line 2} follows by bringing factors that do not depend on $b^x$ outside of the sum.  Line \eqref{consistency line 3} follows by the induction hypothesis that $Q_b$ is consistent for all $b\subseteq S$.  And line \eqref{consistency line 4} follows by the recursive expression of \eqref{eq:branching Markov kernel}.

Consistency requires that $\sum_{t''\in D^{-1}_{S,S^x}(t')}Q_{S^x}(t^*,t'')=Q_S(t,t')$ for all $t^*\in D^{-1}_{S,S^x}(t)$ and hence we must have 
$$\frac{p_{S^x}(\Pi_{t^*},{\bf e}_{S^x})}{q_{S^x}(\Pi_{t^*},{\bf1}_{S^{x}})}+\sum_{\pi^*\in D^{-1}_{S,S^x}(\Pi_{t'})}\frac{p_{S^x}(\Pi_{t^*},\pi^*)}{q_{S^x}(\Pi_{t^*},{\bf1}_{S^x})}\frac{q_{S}(\Pi_t,{\bf1}_{S})}{p_S(\Pi_t,\Pi_{t'})}=1$$
above, which is equivalent to 
$$p_{S^x}(\Pi_{t^*},{\bf e}_{S^x})+\frac{q_{S}(\Pi_t,{\bf1}_S)}{p_S(\Pi_t,\Pi_{t'})}\sum_{\pi^*\in D^{-1}_{S,S^x}(\Pi_{t'})}p_{S^x}(\Pi_{t^*},\pi^*)=q_{S^x}(\Pi_{t^*},{\bf1}_{S^x}).$$
Suppose that $\sum_{\pi^*\in D^{-1}_{S,S^x}(\Pi_{t'})}p_{S^x}(\Pi_{t^*},\pi^*)\neq p_S(\Pi_t,\Pi_{t'})$, then 
$$\frac{q_S(\Pi_t,{\bf1}_S)}{p_S(\Pi_t,\Pi_{t'})}\sum_{\pi^*\in D^{-1}_{S,S^x}}p_{S^x}(\Pi_{t^*},\pi^*)\neq q_S(\Pi_t,{\bf1}_S)$$
and 
$$\frac{p_{S^x}(\Pi_{t^*},{\bf e}_{n+1})}{q_{S^x}(\Pi_{t^*},{\bf1}_{S^{*}})}+\sum_{\pi^*\in D^{-1}_{S,S^x}(\Pi_{t'})}\frac{p_{S^x}(\Pi_{t^*},\pi^*)}{q_{S^x}(\Pi_{t^*},{\bf1}_{S^x})}\frac{q_{S}(\Pi_t,{\bf1}_{S})}{p_S(\Pi_t,\Pi_{t'})}\neq1$$
by the assumption that $p_T(\cdot,\cdot)$ is consistent for all $T\subseteq S$ and our additional assumption. 
%$$p_{S^
Hence, we conclude that consistency of $Q_S$ and $Q_{S^x}$, along with our additional assumption, implies that $p_S(\pi,\cdot)$ and $p_{S^x}(\pi^*,\cdot)$ are consistent for all $\pi\in\mathcal{P}_S$ with $\#\pi>1$ and $\pi^*\in D^{-1}_{S,S^x}(\pi)$.

Reversal of the above argument shows that consistency of $p_S(\pi,\cdot)$ for $\pi$ with $\#\pi>1$ is enough for $Q_S$ to be consistent in \eqref{consistency line 4}.
\end{proof}
A priori, it is not obvious that either of the implications in the above theorem must hold, and it is potentially useful to know that a consistent family of partition-valued transition kernels is sufficient to construct a consistent family of tree-valued processes by the AB Algorithm, provided that $p(B,{\bf1})<1$ for every $B$.  
\paragraph{Infinitely exchangeable kernels}
A tree-valued process $(T_j,j\geq1)$ on $\mathcal{T}$ is {\em infinitely exchangeable} if its finite-dimensional distributions are both finitely exchangeable for every $n\geq1$ and consistent.  More precisely, for each $n\geq1$ let $F_n$ be a probability measure on $\treesn$ and let $F:=(F_n,n\geq1)$ be the family of finite-dimensional distributions on $(\treesn,n\geq1)$.  The collection of spaces $(\treesn,n\geq1)$ forms a projective system, i.e.\ for every $m\leq n$ and injection map $\varphi_{m,n}:[m]\rightarrow[n]$, there is an associated projection $\varphi_{m,n}^*:\treesn\rightarrow\mathcal{T}_m$.  The collection of finite-dimensional measures $F$ is {\em infinitely exchangeable} if for each $m\leq n$ and injection map $\varphi_{m,n}$
$$F_m\equiv F_n\varphi^{*-1}_{m,n}.$$
That is, the measure induced on $\treesn$ by $\varphi_{m,n}$, $F_n\varphi^{*-1}_{m,n}$, corresponds to $F_m$.

A family of Markov kernels $\{p_n(\cdot,\cdot),n\geq1\}$ is infinitely exchangeable if $p_m(t,\cdot)=p_n(t^*,\varphi_{m,n}^{*-1}(\cdot))$ for all $m\leq n$ and injection maps $\varphi_{m,n}:[m]\rightarrow[n]$ \cite{BurkeRosenblatt1958}.  Putting together theorems \ref{thm:exchangeable branching kernel} and \ref{thm:consistent branching kernel} we arrive at a condition for the infinite exchangeability of $Q$ in terms of associated partition-valued Markov kernels.  In particular, if $\{p_S:S\subseteq A\}$ are finitely exchangeable and consistent, and $p_S(\cdot,{\bf1}_S)<1$ for every $S$, then $Q$ is infinitely exchangeable and there is a unique transition measure $Q^{\infty}$ on $\mathcal{T}$, the space of fragmentation trees of $\mathbb{N}$, such that for every $n\geq1$ and $t,t'\in\treesn$,
$$Q_n^{\infty}(t,t')=Q^{\infty}(t^{\infty},\{t^*\in\mathcal{T}:t^*_{|[n]}=t'\})$$
for any $t^{\infty}\in\{t^*:t^*_{|[n]}=t\}$.
The coalescent process does not satisfy this condition because it becomes absorbed in the one-block state almost surely, but other known processes do, e.g.\ exchangeable fragmentation-coalescence (EFC) processes \cite{Berestycki2004} and $\varrho_{\nu}$-Markov processes \cite{Crane2011a}.  We now turn our attention to the $\varrho_{\nu}$-Markov process.

\section{Cut-and-Paste algorithm}\label{section:cut-and-paste algorithm}
We now consider an algorithm for generating a random sequence of set partitions.  A special realization of this algorithm has been presented in \cite{Crane2011a}, which is called the $\varrho_{\nu}$-Markov process for its connection to the paintbox process of Kingman \cite{Kingman1978a}.  Here we outline a more general algorithm, which we call the {\em cut-and-paste} (CP) algorithm, and we shall henceforth refer to the aforementioned $\varrho_{\nu}$-process by the more descriptive title of {\em cut-and-paste} process with parameter $\nu$, or $\cp(\nu)$ process.  

For $A\subseteq\mathbb{N}$, let $\{P_b:b\subseteq A\}$ be a collection of probability measures on $\mathcal{P}_b$ for each $b\subseteq A$, and let $\mu$ denote a probability measure on an at most countable set of labels, which we without loss of generality take to be the set of natural numbers $\mathbb{N}$.  Given a set partition $\pi:=\{\pi_1,\ldots,\pi_k\}\in\mathcal{P}_A$, we generate $\pi'\in\mathcal{P}_A$ by the cut-and-paste algorithm as follows.
\paragraph{Cut-and-Paste (CP) Algorithm}

\vspace{3mm}
\begin{itemize}
	\item[(i)] Generate independent random partitions $C_1,\ldots,C_k$, where for each $i=1,\ldots,k$, $C_i:=\{C_{i,1},\ldots,C_{i,k_i}\}\sim P_{\pi_i}$ is a random partition of block $\pi_i$ of $\pi$, and we list the blocks of $C_i$ in order of appearance.
	\item[(ii)] Generate independent random permutations $\sigma_1,\ldots,\sigma_k$, where for each $i=1,\ldots,k$ $\sigma_i$ is a uniform random permutation of $[\#C_i]=[k_i]$.
	\item[(iii)] Independently for each $i=1,\ldots,k$, generate $m_i:=(m_{i,1},\ldots,m_{i,k_i})$ by drawing without replacement from $\mu$ (a size-biased ordering of the atoms of $\mu$) and assigning label $m_{i,\sigma_i(j)}$ to block $C_{i,j}$ of $C_i$.
	\item[(iv)] For each $l\in\mathbb{N}$, put $\pi'_{l}=\{C_{ij}:m_{i\sigma_i(j)}=l\}$, the collection of blocks of $C_1,\ldots,C_k$ which are labeled $l$ in step (iii).
	\item[(v)] Put $\pi':=\{\pi'_{l}:l\in\mathbb{N}\}\backslash\{\emptyset\}$, the non-empty collection of $\pi'_{l}$ from step (v).
\end{itemize}
 
The name cut-and-paste is derived from steps (i) and (iv) of this algorithm which involve, respectively, cutting (partitioning) the blocks of $\pi$ independently according to some measure and then pasting (coagulating) blocks which are assigned the same label in step (iii).  This procedure can be synthesized in the form of a $k\times\mathbb{N}$ matrix, a generalization of the matrix construction of the $\cp(\nu)$ process in \cite{Crane2011a}, for any $k=1,2,\ldots,\infty$, as follows.

Let $\pi,C_1,\ldots,C_k,\sigma_1,\ldots,\sigma_k,m_1,\ldots,m_k$ be as above.  Write $m\sigma_i(j):=m_{i,\sigma(j)}$ and $(m\sigma_i)^{-1}(l):=\{j:m_{i,\sigma_i(j)}=l\}$.  If $(m\sigma_i)^{-1}(l)=\emptyset$ then we write $C_{i,(m\sigma_i)^{-1}(l)}=\emptyset$ in what follows. Then put $\pi'$ equal to the non-empty column totals of the matrix
\begin{displaymath}
\bordermatrix{\text{}&C_{\subdot 1} & C_{\subdot 2} & \ldots & C_{\subdot j} & \ldots\cr
\pi_1 & C_{1,(m\sigma_1)^{-1}(1)} & C_{1,(m\sigma_1)^{-1}(2)} &\ldots& C_{1,(m\sigma_1)^{-1}(j)} & \ldots\cr
\pi_2 & C_{2,(m\sigma_2)^{-1}(1)} & C_{2,(m\sigma_2)^{-1}(2)} &\ldots& C_{2,(m\sigma_2)^{-1}(j)} & \ldots\cr
\vdots & \vdots & \vdots & \ddots & \vdots\cr
\pi_k & C_{k,(m\sigma_k)^{-1}(1)} & C_{k,(m\sigma_k)^{-1}(2)} &\ldots& C_{k,(m\sigma_k)^{-1}(j)}& \ldots\cr}.
\end{displaymath}
That is $\pi':=\{\pi'_l:l=1,2,\ldots\}\backslash\{\emptyset\}$ where $\pi'_l=\bigcup_{i=1}^k C_{i,(m\sigma_i)^{-1}(l)}$ for each $l=1,2,\ldots.$

The above procedure is pretty flexible, and the full extent of processes which are generated in this way remains to be seen.  A particular process which arises according to a special case of the CP algorithm is the $\cp(\nu)$ process where the measure $\mu$ is assumed to be the uniform distribution on $[k]$ for some $k\geq1$.  This generates a tree-valued process, the {\em cut-and-paste ancestral branching} process, with some special properties, which we now discuss.
\section{Cut-and-paste ancestral branching processes}\label{section:cut-and-paste ancestral branching}
Above, we have studied the general formulation of both the ancestral branching algorithm on $\mathcal{T}$ and cut-and-paste algorithm on $\mathcal{P}$ and have shown some general relationships between exchangeable and consistent Markov kernels on partitions and their corresponding AB kernels on $\mathcal{T}$.  We now turn our attention to a particular family of partition-valued Markov processes which we previously studied in \cite{Crane2011a}.  First, we discuss some preliminaries.

Let $\masspartition=\{(s_1,s_2,\ldots):s_1\geq s_2\geq\ldots\geq0,\mbox{ }\sum_{i}s_i\leq1\}$ be the space of {\em ranked-mass partitions}.  For $s\in\masspartition$, let $X:=(X_1,X_2,\ldots)$ be independent random variables with distribution 
\begin{displaymath}\mathbb{P}_s(X_i=j)=\left\{\begin{array}{cc}
s_j, & j\geq1\\
1-\sum_{k=1}^{\infty}s_k, & j=-i\\
0,&\mbox{otherwise.}\end{array}\right.\end{displaymath}  The partition $\Pi(X)$ generated by $s$ through $X$ satisfies $i\sim_{\Pi(X)}j$ if and only if $X_i=X_j.$  The distribution of $\Pi(X)$ is written $\varrho_s$ and $\Pi(X)$ is called the {\em paintbox based on $s$}.  For a probability measure $\nu$ on $\masspartition$, the paintbox based on $\nu$ is the $\nu$-mixture of paintboxes, written $\varrho_{\nu}(\cdot):=\int_{\masspartition}\varrho_{s}(\cdot)\nu(ds).$  Any partition obtained in this way is an exchangeable random partition of $\mathbb{N}$ and every infinitely exchangeable partition admits a representation as the paintbox generated by some $\nu$. See \cite{BertoinPIMS} and \cite{Pitman2005} for more details on the paintbox process.

For any probability measure $\nu$ on $\masspartitionk:=\{s\in\masspartition:s_j=0\mbox{ }\forall j>k,\mbox{ }\sum s_j = 1\}$, the {\em ranked $k$-simplex}, let $\varrho_{\nu}(\cdot)$ be the paintbox based on $\nu$ as described above.  For each $n\geq1$, define finite-dimensional transition probabilities on $\partitionsnk$ by
\begin{equation}p_n(B,B';\nu):=\frac{k!}{(k-\#B')!}\prod_{b\in B}\frac{(k-\#B'_{|b})!}{k!}\varrho_{\nu}(B'_{|b}).\label{eq:fidi tps}\end{equation}
The collection $(p_n(\cdot,\cdot;\nu),n\geq1)$ of transition probabilities characterizes an infinitely exchangeable Markov process on $\mathcal{P}^{(k)}$, called the {\em cut-and-paste} process with parameter $\nu$, $\cp(\nu)$-process, under the usual deletion operation $D_{n,n+1}:\mathcal{P}_{[n+1]}\rightarrow\partitionsn$, $B\mapsto D_{n,n+1}(B):=B_{|[n]}$ \cite{Crane2011a}.

The transition mechanism on $\partitionsk$ characterized by the finite-dimensional transition probabilities in \eqref{eq:fidi tps} admits the following useful construction.  Let $B\in\partitionsk$, $C:=(C_1,\ldots,C_k)$ be i.i.d.\ $\varrho_{\nu}$ paintboxes and $\sigma:=(\sigma_1,\ldots,\sigma_k)$ be i.i.d.\ uniform random permutations of $[k]$.  Construct the matrix
\begin{displaymath}
\bordermatrix{\text{}&C_{\subdot 1} & C_{\subdot 2} & \ldots & C_{\subdot k}\cr
B_1 & C_{1,\sigma_1(1)}\cap B_1 & C_{1,\sigma_1(2)}\cap B_1&\ldots& C_{1,\sigma_1(k)}\cap B_1\cr
B_2 & C_{2,\sigma_2(1)}\cap B_2 & C_{2,\sigma_2(2)}\cap B_2&\ldots& C_{2,\sigma_2(k)}\cap B_2\cr
\vdots & \vdots & \vdots & \ddots & \vdots\cr
B_k & C_{k,\sigma_k(1)}\cap B_k & C_{k,\sigma_k(2)}\cap B_k & \ldots & C_{k,\sigma_k(k)}\cap B_k}=:B\cap C^{\sigma}.
\end{displaymath}
We write $\cp(B,C,\sigma):=\left\{\bigcup_{j=1}^k(B_j\cap C_{j,\sigma_j(i)}),1\leq i\leq k\right\}\backslash\emptyset$ to be the partition whose blocks are given by the column totals of $B\cap C^{\sigma}$.  This formulation corresponds to the finite-dimensional transitions in \eqref{eq:fidi tps} and can be used in an alternate specification of the ancestral branching algorithm based on these transition probabilities.
%\subsection{$\varrho_{\nu}$-branching Markov chain on $\mathcal{T}^{(k)}$}\label{section:paintbox branching Markov chain}

For $n\geq1$, $k\geq2$ and $\nu$ a probability measure on $\masspartitionk$, let $p_n(\cdot,\cdot;\nu)$ denote the $\cp(\nu)$ transition probability on $\partitionsnk$ in \eqref{eq:fidi tps} and $q_n(\cdot,\cdot;\nu)=1-p_n(\cdot,\cdot;\nu)$ its complementary probability.  The family $\{p_n(\cdot,\cdot;\nu):n\geq1\}$ is infinitely exchangeable and so defines a unique transition probability $p_A(\cdot,\cdot;\nu)$ on $\mathcal{P}_A^{(k)}$ for each $A\subset\mathbb{N}$ by 
$$p_A(\cdot,\cdot;\nu):=p_{\#A}(\cdot,\cdot;\nu)$$
for $\#A<\infty$ and $p_A(\cdot,\cdot;\nu)=p_{\mathbb{N}}(\cdot,\cdot;\nu)$ otherwise.  

Furthermore, for $\nu$ non-degenerate at $(1,0,\ldots,0)$ we have that $p_b(\cdot,{\bf1}_b;\nu)<1$ for all $b\subset\mathbb{N}$ with $\#b>1$, and so \eqref{eq:branching Markov kernel} is well-defined and the results of section \ref{section:infinitely exchangeable processes} hold.  In particular, the $\mathcal{T}$-valued process induced by the finite-dimensional transition probabilities \eqref{eq:fidi tps} and the ancestral branching algorithm is infinitely exchangeable.

For the AB algorithm based on the transition probabilities of the $\cp(\nu)$ process, we can describe an alternative, though equivalent, formulation which is helpful in later sections.
\subsection{Alternative construction of the cut-and-paste ancestral branching Markov chain}\label{section:alternative construction}
We introduce a {\em genealogical indexing system} to label the elements of $t_A\in\trees_A$ (chapter 1.2.1 of Bertoin \cite{Bertoin2006}) as follows.

We write $$\mathcal{U}:=\bigcup_{n=0}^{\infty}\mathbb{N}^n$$ to denote the infinite set of all indices, with convention that $\mathbb{N}^0=\{\emptyset\}$.

For a fragmentation tree $T$, the $n$th generation of $T$ is the collection of children $t\in T$ such that $\#\anc(t)=n-1$.  For each $u=(u_1,\ldots,u_n)\equiv u_1u_2\cdots u_n\in\mathcal{U}$, $n$ is the generation of $u$.  Write $u-:=(u_1,\ldots,u_{n-1})$ to denote the parent of $u$ and $ui:=(u,i):=(u_1,\ldots,u_n,i)$ for the $i$th child of $u$.  As we are working in the context of fragmentations of subsets of $\mathbb{N}$, the $i$th child of $t\in T$ is the $i$th child to appear in a list when the elements of $\frag(t)$, the children of $t$, are listed in order of their least element.

A Markov chain on $\mathcal{T}^{(k)}$ which is governed by the same transition law as in the previous section can be constructed by a {\em genealogical branching procedure} as follows.

Let $k\geq2$ and $\nu$ be a probability measure on $\masspartitionk$ which is non-degenerate at $(1,0,\ldots,0)$.  For $T,T'\in\treesk$, the transition $T\mapsto T'$ occurs as follows.  Generate $\{B^{u}: u\in\mathcal{U}\}$ i.i.d.\ $\varrho_{\nu}^{(k)}$ partition sequences, where $\varrho_{\nu}^{(k)}:=\varrho_{\nu}\otimes\cdots\otimes\varrho_{\nu}$ is the product measure of paintboxes based on $\nu$, and $\{\sigma^{u}: u\in\mathcal{U}\}$ i.i.d.\ $k$-tuples of i.i.d.\ uniform permutations of $[k]$.

{\bf Genealogical Branching Procedure}
\begin{itemize}
	\item[(i)] Put $\Pi_{T'}=\cp(\Pi_T,B^{\emptyset},\sigma^{\emptyset})$, the partition obtained from the column totals of $\Pi_T\cap (B^{\emptyset})^{\sigma^{\emptyset}}$, as shown in section \ref{section:cut-and-paste algorithm};
	\item[(ii)] for $A^{u}\in T'$, put $A^{uj}$ equal to the $j$th block of $\cp(\Pi_{T_{|A^{u}}},B^{u},\sigma^{u})$ listed in order of least elements.
\end{itemize}
In other words, each $B^u$ is an independent $k$-tuple of independent paintboxes based on $\nu$ and we index this sequence just as we index the vertices of a tree.  Likewise, each $\sigma^u$ is an independent $k$-tuple $(\sigma^u_1,\ldots,\sigma^u_k)$ of i.i.d.\ uniform permutations of $[k]$.  The next state $T'$ is obtained from $T$ by a sequential branching procedure which starts from the root and progressively branches the roots of the subtrees restricted to each child of $T'$.
The children of $T'$ are given by $\{A^{u},u\in\mathcal{U}\}$ and for each $n\geq1$ the restriction to $[n]$ of $T'$ is $T'_{|[n]}=\{A^{u}\cap[n],u\in\mathcal{U}\}$.

The genealogical branching procedure simultaneously generates sequences of trees on $\treesn$ for every $n\geq1$.  It should be plain that this construction is equivalent to that in section \ref{section:cut-and-paste ancestral branching} since it uses the matrix construction of the $\cp(\nu)$ transition probabilities on $\mathcal{P}_A^{(k)}$.  The benefit to this construction is that it gives an explicit recipe which will be employed in the proofs of various properties of this process in later sections.  For completeness, we provide a proof that the finite-dimensional transition probabilities of this process coincide with \eqref{eq:fidi tps trees}.
\begin{prop}\label{prop:fidi tree tps}Let $T\mapsto T'\in\treesk$ be a transition generated by the above genealogical branching procedure.  For $n\geq1$, the finite-dimensional transition probability of the restricted transition $T_{|[n]}\mapsto T'_{|[n]}$ is
\begin{equation}Q_n(T,T';\nu):=\prod_{b\in T'}\frac{p_b(\Pi_{T_{|b}},\Pi_{T'_{|b}};\nu)}{q_b(\Pi_{T_{|b}},{\bf1}_{b};\nu)}.\label{eq:fidi tps trees}\end{equation}\end{prop}
\begin{proof}
Write $p_n(\cdot,\cdot)\equiv p_n(\cdot,\cdot;\nu)$ and $q_n(\cdot,\cdot)\equiv q_n(\cdot,\cdot;\nu)$.  For $n\geq1$, the branching of the root of $T'_{|[n]}$ given $T_{|[n]}$ is given by $A^{u(m)}_{|[n]}$ for $u(m)\in\mathcal{U}$ such that $u(m)=(\underbrace{1,\ldots,1}_{m\mbox{ times}},0,\ldots)$ and $m$ is the smallest $m\geq1$ such that $A^{u(m)}_{|[n]}\notin\{[n],\emptyset\}$, i.e.\ the first non-trivial partition of $[n]$ obtained by the above procedure.  The distribution of the branching of the root of $T'_{|[n]}$ given $T_{|[n]}$ obtained in this way is 
$$\sum_{i=0}^{\infty}p_n(\Pi_{T_{|[n]}},\Pi_{T'_{|[n]}})p_n(\Pi_{T_{|[n]}},{\bf 1}_n)^i=\frac{p_n(\Pi_{T_{|[n]}},\Pi_{T'_{|[n]}})}{q_n(\Pi_{T_{|[n]}},{\bf1}_n)}.$$

By independence of the steps of the procedure, we can write the distribution of the transition $T\mapsto T'$ recursively as 
$$\pi_n(T,T')=\frac{p_n(\Pi_{T},\Pi_{T'})}{q_n(\Pi_{T},{\bf1}_n)}\prod_{b\in\Pi_{T'}}\pi_b(T_{|b},T'_{|b}).$$
Iterating the above argument yields \eqref{eq:fidi tps trees}.
\end{proof}
\subsection{Equilibrium measure}\label{section:equilibrium measure}
The form of $Q_n(T,T';\nu)$ in \eqref{eq:fidi tps trees} is a product of independent transition probabilities of the branching at the root in each of the subtrees of $T'$.  It is known that for $\nu$ non-degenerate at $(1,0,\ldots,0)\in\masspartitionk$, $p_n(\cdot,\cdot;\nu)$ has a unique equilibrium distribution for each $n\geq1$ \cite{Crane2011a}.  Since $p_n(B,B';\nu)>0$ for every $n\geq1$ and $B,B'\in\partitionsnk$, we have that $Q_n(t,t';\nu)>0$ for all $t,t'\in\treesnk$ and so each $Q_n(\cdot,\cdot;\nu)$ is aperiodic and irreducible for non-degenerate $\nu\in\masspartitionk$.  The following proposition is immediate.
\begin{prop}\label{prop:existence fidi stationary}Let $\nu$ be a probability measure on $\masspartitionk$ such that $\nu((1,0,\ldots,0))<1$ and let $Q_n(\cdot,\cdot;\nu)$ be the $\cp(\nu)$-ancestral branching Markov kernel, then there exists a unique measure $\rho_n(\cdot;\nu)$ on $\treesnk$ which is stationary for $Q_n(\cdot,\cdot;\nu)$ for each $n\geq1$.\end{prop}
It is easy to see that the above proposition can be generalized to general Markov chains by modifying the above condition on $\nu\neq(1,0,\ldots,0)$ to state $p_n(B,B)>0$ for every $n\geq1$ and $B\in\partitionsn$ and $p_n(\cdot,\cdot)$ is irreducible for every $n\geq1$.

The existence of $\rho_n(\cdot;\nu)$ and the finite exchangeability and consistency of $Q_n(\cdot,\cdot;\nu)$ for each $n\geq1$ induce finite exchangeability and consistency for the collection $(\rho_n(\cdot;\nu),n\geq1)$ of equilibrium measures.
\begin{prop}\label{prop:inf exch stationary}Let $(Q_n(\cdot,\cdot),n\geq1)$ be an infinitely exchangeable collection of ancestral branching Markov kernels \eqref{eq:branching Markov kernel} on $(\treesn,n\geq1)$ and suppose for each $n\geq1$ $\rho_n(\cdot)$ is a unique stationary distribution for $Q_n(\cdot,\cdot)$.  Then the family $(\rho_n(\cdot),n\geq1)$ is infinitely exchangeable.\end{prop}
\begin{proof}
For $T''\in\mathcal{T}_{n+1}$
$$\rho_{n+1}(T'')=\sum_{T^*\in\mathcal{T}_{n+1}}\rho_{n+1}(T^*)Q_{n+1}(T^*,T'')$$
by stationarity.

Let $T'\in\treesn$ and for $1\leq m\leq n$ let $\varphi:[m]\rightarrow[n]$ be an injection with associated projection $\varphi^*:\treesn\rightarrow\mathcal{T}_m$.  Then
\begin{eqnarray}
\underbrace{\sum_{T''\in \varphi^{*-1}(T')}\rho_{n}(T'')}_{(\rho_{n}\varphi^{*-1})(T')}&=&\sum_{T''\in \varphi^{*-1}(T')}\sum_{T^*\in\mathcal{T}_{n}}\rho_{n}(T^*)Q_{n}(T^*,T'')\label{prop:consistent stationary:line1}\\
&=&\sum_{T\in\mathcal{T}_m}\sum_{T^*\in \varphi^{*-1}(T)}\rho_{n}(T^*)\underbrace{\left[\sum_{T''\in \varphi^{*-1}(T')}Q_{n}(T^*,T'')\right]}_{Q_n\varphi^{*-1}(T,T')\equiv Q_m(T,T')}\label{prop:consistent stationary:line2}\\
&=&\sum_{T\in\mathcal{T}_m}Q_m(T,T')\underbrace{\sum_{T^*\in \varphi^{*-1}(T)}\rho_{n}(T^*)}_{\rho_n\varphi^{*-1}(T)}\label{prop:consistent stationary:line3}\\
&=&\sum_{T\in\mathcal{T}_m}(\rho_{n}\varphi^{*-1})(T)Q_m(T,T').\label{prop:consistent stationary:line4}
\end{eqnarray}
The expression in \eqref{prop:consistent stationary:line2} follows from \eqref{prop:consistent stationary:line1} by changing the order of summation and noting that each $T^*\in\mathcal{T}_{n}$ corresponds to exactly one $T\in\mathcal{T}_m$ through the mapping $\varphi^{*}$; \eqref{prop:consistent stationary:line3} follows from \eqref{prop:consistent stationary:line2} by the consistency of $Q_n(\cdot,\cdot)$ for each $n\geq1$; and \eqref{prop:consistent stationary:line4} follows \eqref{prop:consistent stationary:line3} by the definition of induced measures.
Hence, the induced measure $\rho_{n}\varphi^{*-1}$ is stationary for $Q_n$.  By uniqueness, $\rho_{n}\varphi^{*-1}\equiv\rho_m$ for every injective mapping $\varphi:[m]\rightarrow[n]$.  Hence, $(\rho_n, n\geq1)$ is an infinitely exchangeable family of measures on $(\mathcal{T}_n,n\geq1).$
\end{proof}
The existence of an infinitely exchangeable equilibrium measure $\rho(\cdot)$ on $\mathbb{N}$-labeled trees, $\trees$, is a direct consequence of the finite exchangeability and consistency of the system $(\rho_n(\cdot),n\geq1)$ shown in proposition \ref{prop:inf exch stationary} and Kolmogorov's extension theorem \cite{Billingsley1995}.  In this case, the measure $\rho(\cdot)$ satisfies
$$\rho_n(T_n)=\rho\left(\{T\in\trees:T_{|[n]}=T_n\}\right)$$
for every $n\geq1$.

The above results for the equilibrium measure $\rho(\cdot)$ apply specifically to the $\cp(\nu)$ ancestral branching process under the condition that $\nu$ is non-degenerate at $(1,0,\ldots,0)\in\masspartitionk$.
\begin{cor}\label{cor:existence inf exch stationary}For $\nu$ non-degenerate at $(1,0,\ldots,0)\in\masspartitionk$, the collection of stationary measures $(\rho_n(\cdot;\nu),n\geq1)$ in proposition \ref{prop:existence fidi stationary} is infinitely exchangeable.\end{cor}
Although the existence of a unique stationary measure on $\mathcal{T}^{(k)}$ is implicit in the construction of the transition at the beginning of this section, the form of the finite-dimensional and infinite-dimensional stationary measure remains unknown.  Note that, though the transition probabilities \eqref{eq:branching Markov kernel} are conditionally of {\em fragmentation type}, i.e.\ given $T$ and $b\in T'$ the children of $b$ are distributed independently of the rest of $T'$, the equilibrium measure need not be of this form.  Furthermore, it is of interest whether or not some subclass of the $\cp(\nu)$ ancestral branching Markov chains is reversible and, if so, under what conditions this property holds.
\subsubsection{Continuous-time ancestral branching process}\label{section:continuous-time}
An infinitely exchangeable collection $(Q_n,n\geq1)$ of ancestral branching transition probabilities can be embedded in continuous time in a straightforward way by defining the Markovian infinitesimal jump rates $r_n(\cdot,\cdot)$ on $\treesn$ by
\begin{eqnarray}
		r_n(T,T')=\left\{\begin{array}{cc}
													\lambda Q_n(T,T'), & T\neq T'\\
													0, & \mbox{otherwise,}
												\end{array}\right.\label{eq:rates}
	\end{eqnarray}
for some $\lambda>0$.
	
%The infinitesimal generator $Q_n^{\nu}$ of the process on $\treesnk$ governed by $r_n$ has entries
%\begin{eqnarray}
%		Q_n^{\nu}(T,T';\nu)=\lambda\times\left\{\begin{array}{cc}
%																	\pi_n(T,T';\nu), & T\neq T'\\
%																	\pi_n(T,T;\nu)-1, & T=T'.
%																\end{array}\right.\label{eq:infinitesimal generator}
%																\end{eqnarray}
\begin{defn}\label{defn:process}A process $T:=(T(t),t\geq0)$ is an ancestral branching Markov process if for each $n\geq1$, the restriction $T_{|[n]}:=(T_{|[n]}(t),t\geq0)$ is a Markov process on $\treesn$ with infinitesimal transition rates $r_n(\cdot,\cdot)$.\end{defn}
A process on $\mathcal{T}$ whose finite-dimensional restrictions are governed by $r_n$ can be constructed by running a Markov chain on $\treesn$ governed by \eqref{eq:fidi tps trees} in which only transitions $T\mapsto T'$ for $T\neq T'$ are permitted, and adding a hold time which is exponentially distributed with mean $1/[\lambda-\lambda r_n(T,T)]$.  The following proposition is a corollary of theorems \ref{thm:exchangeable branching kernel} and \ref{thm:consistent branching kernel}.
\begin{cor}\label{cor:consistent Q-matrix}For measure $\nu$ on $\masspartitionk$, the collection $(R_n^{\nu},n\geq1)$ of finite-dimensional $Q$-matrices based on \eqref{eq:rates} are consistent.\end{cor}
%\begin{proof}
%Fix $n\geq1$ and let $T,T'\in\treesnk$ such that $T\neq T'$ and let $T^*\in D_{n,n+1}^{-1}(T)$.  Then
%\begin{eqnarray*}
%\sum_{T''\in D_{n,n+1}^{-1}(T')}Q_{n+1}^{\nu}(T^*,T'')&=&\sum_{T''\in D_{n,n+1}^{-1}(T')}r_{n+1}(T^*,T'';\nu)\\
%&=&\sum_{T''\in D_{n,n+1}^{-1}(T')}\lambda\pi_{n+1}(T^*,T'';\nu)\\
%&=&\lambda\pi_n(T,T',\nu)\\
%&=&Q_n^{\nu}(T,T'),
%\end{eqnarray*}
%by the consistency of $\pi_n(T,\cdot;\nu)$ for each $n\geq1$.
%
%For $T=T'$ and $T^*\in D^{-1}_{n,n+1}(T)$
%\begin{eqnarray*}
%\lefteqn{\sum_{T''\in D_{n,n+1}^{-1}(T)}Q_{n+1}^{\nu}(T^*,T'')=}\\
%&=&Q_{n+1}^{\nu}(T^*,T^*) + \sum_{T''\in D_{n,n+1}^{-1}(T)\backslash\{T^*\}}r_{n+1}(T^*,T'';\nu)\\
%&=&\lambda\left(\pi_{n+1}(T^*,T^*;\nu)-1+\sum_{T''\in D_{n,n+1}^{-1}(T)\backslash\{T^*\}}\pi_{n+1}(T^*,T'';\nu)\right)\\
%&=&\lambda\left(\sum_{T''\in D_{n,n+1}^{-1}(T)}\pi_{n+1}(T^*,T'';\nu)-1\right)\\
%&=&\lambda\left(\pi_n(T,T;\nu)-1\right)\\
%&=&Q_n^{\nu}(T,T).
%\end{eqnarray*}
%\end{proof}
The existence of a continuous-time process with embedded jump chain governed by \eqref{eq:fidi tps trees} is now clear by the corollary \ref{cor:consistent Q-matrix} and the discussion at the end of section \ref{section:infinitely exchangeable processes}.
\begin{thm}\label{thm:existence continuous-time}There exists a continuous-time Markov process $(T(t),t\geq0)$ on $\treesk$ governed by $Q^{\nu}$ such that $$Q_n^{\nu}(T,T')=Q^{\nu}(T^{\infty},\{T''\in\treesk:T''_{|[n]}=T'\}),$$ for each $T^{\infty}\in\{T^*\in\treesk:T^*_{|[n]}=T\}$.\end{thm}
\begin{proof}
Corollary \ref{cor:consistent Q-matrix} establishes that the finite-dimensional infinitesimal jump rates $(r_n,n\geq1)$ is finitely exchangeable and consistent.  Kolmogorov's extension theorem implies the existence of $R$ with finite-dimensional restrictions given by $(r_n,n\geq1)$.  Furthermore, for each $n\geq1$ and $T\in\treesn$, $1-r_n(T,T)=\lambda(1-Q_n(T,T))<\lambda<\infty$ so that the finite-dimensional paths are c\`adl\`ag for each $n$, which implies the paths of $(T(t),t\geq0)$ governed by $R$ are c\`adl\`ag.
\end{proof}
The transition rates above are defined in terms of a collection of infinitely exchangeable transition probabilities $(Q_n(\cdot,\cdot),n\geq1)$.  If $Q_n$ has unique equilibrium measure $\rho(\cdot)$, then so does its associated continuous-time process.  We have the following corollary for the stationary measure of the continuous-time process.
\begin{cor}Let $(T(t),t\geq0)$ be a continuous-time process governed by an infinitely exchangeable collection $(Q_n,n\geq1)$ of ancestral branching transition probabilities \eqref{eq:branching Markov kernel}.  Further suppose that for each $n\geq1$, $Q_n$ has unique equilibrium measure $\rho_n$ and the characteristic measure $Q$ on $\mathcal{T}$ has unique equilibrium measure $\rho$ as in proposition \ref{prop:inf exch stationary}.  Then $(T(t),t\geq0)$ has unique equilibrium measure $\rho$.\end{cor}
For our purposes, we now restrict our attention to the $\cp(\nu)$ subfamily of ancestral branching processes on $\treesk$ with $\nu$ some measure on $\masspartitionk$ for some $k\geq1$.  We index transition measures and stationary measures by $\nu$ to make this explicit.  As we show, the $\cp(\nu)$ associated ancestral branching process is a Feller process and has an associated mass fragmentation process.
%\begin{cor}\label{cor:continuous-time stationary} For $\nu$ non-degenerate at $(1,0,\ldots,0)\in\masspartitionk$, the continuous-time process $T:=(T(t),t\geq0)$ with finite-dimensional infinitesimal generator in \eqref{eq:infinitesimal generator} has unique equilibrium measure $\rho(\cdot)$ as in theorem \ref{thm:existence inf stationary}.\end{cor}
\subsection{Poissonian construction}\label{section:poissonian construction}
A consequence of the above continuous-time embedding and the alternative specification of the cut-and-paste ancestral branching algorithm given in section \ref{section:alternative construction} is yet another alternative construction via a Poisson point process.

Let $P=\{(t,B^u:u\in\mathcal{U})\}\subset\mathbb{R}^+\times\prod_{u\in\mathcal{U}}\left[\prod_{j=1}^{k}\partitionsk\right]$ be a Poisson point process with intensity measure $dt\otimes\lambda\bigotimes_{u\in\mathcal{U}}\varrho_{\nu}^{(k)}$, where $\varrho_{\nu}^{(k)}$ is the product measure $\varrho_{\nu}\otimes\cdots\otimes\varrho_{\nu}$ on $\prod_{j=1}^k\partitionsk$.  So for each $(t,B^u)\in P$, $B^u:=(B^u_1,\ldots,B^u_k)\in\prod_{j=1}^k\partitionsk$ is distributed as $\varrho_{\nu}^{(k)}$ and is labeled according to the genealogical index system of section \ref{section:alternative construction}.  

Construct a continuous time $\cp(\nu)$-ancestral branching Markov process as follows.  Let $\tau\in\treesk$ be an infinitely exchangeable random fragmentation tree.  For each $n\geq1$, put $T_{|[n]}(0)=\tau_{|[n]}$ and for $t>0$
\begin{itemize}
	\item if $t$ is not an atom time for $P$, then $T_{|[n]}(t)=T_{|[n]}(t-)$;
	\item if $t$ is an atom time for $P$ so that $(t,B^u:u\in\mathcal{U})\in P$, generate $\sigma:=(\sigma^u:u\in\mathcal{U})\in\prod_{u\in\mathcal{U}}\left[\prod_{j=1}^k\mathcal{S}_k\right]$, an i.i.d.\ collection of $k$-tuples of uniform permutations of $[k]$.  Put $T:=T(t-)$ and $T'$ equal to the tree constructed from $T$, $\{B^u:u\in\mathcal{U}\}$ and $\sigma$ through the function $\cp(\cdot,\cdot,\cdot)$ which is described in section \ref{section:alternative construction}.  If $T'_{|[n]}\neq T_{|[n]}$, put $T_{|[n]}(t)=T'_{|[n]}$; otherwise, put $T_{|[n]}(t)=T_{|[n]}(t-)$.
\end{itemize}
\begin{prop}\label{prop:poissonian rates}The above process $T$ is a Markov process on $\treesk$ with transition matrix $Q^{\nu}$ defined by theorem \ref{thm:existence continuous-time}.\end{prop}
\begin{proof}
By the above construction, for every $n\geq1$ and $t>0$, $T_{|[n]}(t)$ evolves according to $r_n^{\nu}$ in \eqref{eq:rates}, $D_{m,n}T_{|[n]}(t)=T_{|[m]}(t)$ for all $m\leq n$, and $T_{|[p]}(t)\in D_{n,p}^{-1}(T_{|[n]}(t))$ for all $p>n$.  Hence, the restriction $T_{|[n]}$ is a $Q_n^{\nu}$-governed Markov process for each $n\geq1$ and the result is clear by consistency of $Q_n^{\nu}$.\end{proof}
By ignoring the arrival times in the above Poissonian construction and looking only at the embedded jump chain, we obtain a discrete-time process which evolves according to the $\cp(\nu)$-ancestral branching algorithm of section \ref{section:ancestral branching algorithm}.
\subsection{Feller process}\label{section:Feller process}
In \cite{Crane2011a} we show that the cut-and-paste process with finite-dimensional Markovian jump rates corresponding to the transition probabilities in \eqref{eq:fidi tps} is a Feller process.  Indeed, we now show that the ancestral branching Markov process on $\mathcal{T}$ which is induced by the $\cp(\nu)$ Markov process is also Fellerian, but we first need some preliminaries.

Define the metric $d:\mathcal{T}\times\mathcal{T}\rightarrow\mathbb{R}^+$ by
\begin{equation}d(T,T'):=1/\max\{n\in\mathbb{N}:T_{|[n]}=T'_{|[n]}\},\label{eq:tree metric}\end{equation}
for every $T,T'\in\mathcal{T}$, with the convention that $1/\infty=0$.
\begin{prop} $d$ is a metric on $\mathcal{T}$.\end{prop}
\begin{proof}
Positivity and symmetry are obvious.  To see that the triangle inequality holds, let $T,T',T''\in\mathcal{T}$ so that $d(T,T')=1/a$ for some $a\geq1$.  Now suppose that $d(T,T'')=1/b\geq1/a$.  Then the triangle inequality is trivially satisfied.  If $d(T,T'')=1/b<1/a$ then $T_{|[b]}=T''_{|[b]}$ for $b>a$ and $T_{|[a]}=T'_{|[a]}$ but $T_{|[a+1]}\neq T'_{|[a+1]}$ by assumption.  Hence, $d(T',T'')=1/a$ and the triangle inequality holds.\end{proof}
\begin{prop}\label{prop:compact space}$(\mathcal{T},d)$ is a compact space.\end{prop}
\begin{proof}
Let $(T^1,T^2,\ldots)$ be a sequence in $\mathcal{T}$.  Any element $T\in\mathcal{T}$ can be written as a compatible sequence of finite-dimensional restrictions, $T:=(T_{|[1]},T_{|[2]},\ldots):=(T_1,T_2,\ldots)$.  The set $\treesn$ is finite for each $n$, and so one can extract a convergent subsequence $(T^{(1)},T^{(2)},\ldots)$ of $(T^1,T^2,\ldots)$ by the diagonal procedure such that $d(T^{(i)},T^{(j)})\leq1/\min\{i,j\}$ for all $i,j$.
\end{proof}
\begin{lemma}\label{lemma:dense set}$C_f:=\{f:\mathcal{T}\rightarrow\mathbb{R}:\exists n\in\mathbb{N}\mbox{ s.t. }d(T,T')\leq1/n\Rightarrow f(T)=f(T')\}$ is dense in the space of continuous functions $\mathcal{T}\rightarrow\mathbb{R}$ under the metric $\rho(f,f'):=\sup_{\tau\in\mathcal{T}}|f(\tau)-f'(\tau)|$.\end{lemma}
\begin{proof}
Let $\varphi:\mathcal{T}\rightarrow\mathbb{R}$ be a continuous function.  Then for every $\epsilon>0$ there exists $n(\epsilon)\in\mathbb{N}$ such that $\tau,\sigma\in\mathcal{T}$ satisfying $d(\tau,\sigma)\leq1/n(\epsilon)$ implies $|\varphi(\tau)-\varphi(\sigma)|\leq\epsilon$.

For fixed $\epsilon>0$, let $N=n(\epsilon)$ and define $f:\mathcal{T}\rightarrow\mathbb{R}$ as follows.  First, partition $\mathcal{T}$ into equivalence classes $\{\tau\in\mathcal{T}:\tau_{|[N]}=t_{|[N]}\}$ for each $t\in\mathcal{T}$.  For each equivalence class $U$, choose a representative element $\tilde{u}\in U$ and put $f(u):=\varphi(\tilde{u})$ for all $u\in U$.  For any $t\in\mathcal{T}$, let $\tilde{t}$ denote the representative of $t$ obtained in this way.  Hence, $f(t)=f(t')=f(\tilde{t})$ for all $t,t'$ such that $d(t,t')\leq 1/N$ and $f\in C_f$.  Thus, $$|f(\tau)-\varphi(\tau)|=|\varphi(\tilde{\tau})-\varphi(\tau)|\leq\epsilon$$ by continuity of $\varphi$ and $$\rho(f,\varphi)=\sup_{\tau}|f(\tau)-\varphi(\tau)|\leq\epsilon,$$ which establishes density.
\end{proof}
Let $\mathbb{P}_t$ be the semi-group of a $\varrho_{\nu}$-branching Markov process $T(\cdot)$, i.e.\ for any continuous $\varphi:\treesk\rightarrow\mathbb{R}$
$$\mathbb{P}_t\varphi(\tau):=\mathbb{E}_{\tau}\varphi(T(t)),$$
the expectation of $\varphi(T(t))$ given $T(0)=\tau.$
\begin{cor}\label{cor:feller}A $\cp(\nu)$-ancestral branching Markov process has the Feller property, i.e.
\begin{itemize}
	\item for each continuous function $\varphi:\treesk\rightarrow\mathbb{R}$, for each $\tau\in\mathcal{P}$ one has 
	$$\lim_{t\downarrow0}\mathbb{P}_t\varphi(\tau)=\varphi(\tau),$$
	\item for all $t>0$, $\tau\mapsto\mathbb{P}_{t}\varphi(\tau)$ is continuous.
\end{itemize}\end{cor}
\begin{proof}
The proof follows the same line of reasoning as corollary 4.2 in \cite{Crane2011a}.  Let $\varphi$ be a continuous function $\treesk\rightarrow\mathbb{R}$.

For $g\in C_f$, $\lim_{t\downarrow0}\mathbb{P}_tg(\tau)=g(\tau)$ is clear since the first jump-time of $T(\cdot)$ is exponential with finite mean.  Denseness of $C_f$ establishes the first point.

For the second point, let $n\geq1$ and $\tau,\tau'\in\treesk$ such that $d(\tau,\tau')<1/n$, i.e.\ $\tau_{|[n]}=\tau'_{|[n]}$.  Use the same Poisson point process $P$, as in section \ref{section:poissonian construction}, to construct $T(\cdot)$ and $T'(\cdot)$ such that $T(0)=\tau$ and $T'(0)=\tau'$.  By construction, $T_{|[n]}=T'_{|[n]}$ and $d(T(t),T'(t))<1/n$ for all $t\geq0$.  Hence, for any continuous $\varphi$, $\tau\mapsto\mathbb{P}_t\varphi_{\tau}$ is continuous.
\end{proof}
By corollary \ref{cor:feller}, we can characterize the $\cp(\nu)$-ancestral branching Markov process $(T(t),t\geq0)$ with finite-dimensional rates $(q_n(\cdot,\cdot;\nu),n\geq1)$ by its infinitesimal generator $\mathcal{G}$ given by
\begin{equation}\mathcal{G}(f)(\tau)=\int_{\treesk}f(\tau')-f(\tau)Q^{\nu}(\tau,d\tau')\notag\end{equation}
for every $f\in C_f$.

Our proof of the Feller property for the $\cp(\nu)$-ancestral branching Markov process makes use of the Poissonian construction of the previous section.  In light of the specification of the cut-and-paste algorithm in section \ref{section:cut-and-paste algorithm}, it is straightforward to see that we can construct a generate cut-and-paste ancestral branching process via a Poisson point process by slight modification, and the various properties shown for the $\cp(\nu)$-ancestral branching process herein may also apply to general cut-and-paste ancestral branching processes.  This is beyond the scope of the current paper, as we are principally interested in establishing properties of the $\cp(\nu)$-ancestral branching process on $\treesk$.
\section{Mass fragmentations}\label{section:mass fragmentations}
A {\em mass fragmentation} of $x\in\mathbb{R}^+$ is a collection $M_x$ of masses such that
\begin{itemize}
	\item[(i)] $x\in M_x$ and
	\item[(ii)] there are $m_1,\ldots,m_k\in M_x$ such that $\sum_{i=1}^k m_i\leq x$ and
	$$M_x=\{x\}\cup M_{m_1}\cup\cdots\cup M_{m_k}.$$
\end{itemize}
We write $\mathcal{M}_x$ to denote mass fragmentations of $x$.  Essentially, a mass fragmentation of $x$ is a fragmentation tree whose vertices are labeled by masses such that the children of a vertex comprise a ranked-mass partition of its parent vertex.  The case where children $\{m_1,\ldots,m_k\}$ of a vertex $m$ satisfy $\sum_{i=1}^k m_i<m$ is called a {\em dissipative} mass fragmentation.  Herein, we are interested in {\em conservative} mass fragmentations which have the property that the children $\{m_1,\ldots,m_k\}$ of every vertex $m\in M_x$ satisfy $\sum_{i=1}^km_i=m$.  It is plain that $\mathcal{M}_x$ is isomorphic to $\mathcal{M}_1$ by scaling, i.e.\ $\mathcal{M}_x=x\mathcal{M}_1$ and so it is sufficient to study $\mathcal{M}_1$.  See Bertoin \cite{Bertoin2006} for a study of Markov processes on $\mathcal{M}_1$ called {\em fragmentation chains}.   Here we construct a Markov process on $\mathcal{M}_1$ which corresponds to the associated mass fragmentation valued process of the $\cp(\nu)$-ancestral branching Markov process on $\treesk$, which has been studied in previous sections.

\begin{defn}\label{defn:asymptotic frequencies}A subset $A\subset\mathbb{N}$ is said to have asymptotic frequency $\lambda$ if
\begin{equation}\lambda:=\lim_{n\rightarrow\infty}\frac{\#(A\cap[n])}{n}\label{eq:asymptotic frequencies}\end{equation}
exists.\end{defn}

A partition $B=\{B_1,B_2,\ldots\}\in\mathcal{P}$ is said to possess asymptotic frequency $\vert\vert B\vert\vert$ if each of its blocks has asymptotic frequency and we write $\vert\vert B\vert\vert:=\left(\vert\vert B_1\vert\vert,\ldots\right)^{\downarrow}\in\masspartition$, the decreasing rearrangement of block frequencies of $B$.  According to Kingman's correspondence \cite{Kingman1980}, any infinitely exchangeable partition $B$ of $\mathbb{N}$ possesses asymptotic frequencies which are distributed according to $\nu$ where $\nu$ is the unique measure on $\masspartition$ such that $B\sim\varrho_{\nu}$.

\subsection{Associated mass fragmentation process}\label{section:associated mass fragmentation process}
Fix $k\geq2$ and let $\nu$ be a probability measure on $\masspartitionk$.  Let $\mathcal{M}_1^{(k)}:=\{\mu\in\mathcal{M}_1:\#A\leq k\mbox{ for every }A\in\mu\}$ be the subspace of conservative mass fragmentations of $1$ such that each $A\in\mu\in\mathcal{M}_1^{(k)}$ has at most $k$ children. 

Construct a Markov chain on $\mathcal{M}_1^{(k)}$ as follows.  For $\mu\in\mathcal{M}_1^{(k)}$, the transition $\mu\mapsto\tilde{\mu}\in\mathcal{M}_1^{(k)}$ is generated by an i.i.d.\ collection $S:=\{s^u:u\in\mathcal{U}\}$ of $\nu^{(k)}$ mass partitions, i.e.\ $s^u:=(s^u_1,\ldots,s^u_k)\in\prod_{i=1}^k\masspartitionk$ is an i.i.d.\ collection of mass partitions distributed according to $\nu$ and $s^w$ is independent of $s^v$ for all $w\neq v$, and $\Sigma:=\{\sigma^u:u\in\mathcal{U}\}$ i.i.d.\ $k$-tuples of i.i.d.\ uniform permutations of $[k]$.
\begin{itemize}
	\item[(i)] Write $\mu:=\{\mu^u:u\in\mathcal{U}\}$ and $\tilde{\mu}:=\{\tilde{\mu}^u:u\in\mathcal{U}\}$.
	\item[(ii)] Put $\tilde{\mu}^{\emptyset}=1$, the root of $\tilde{\mu}$.
	\item[(iii)] Given $\tilde{\mu}^u\in\tilde{\mu}$, put $\tilde{\mu}^{uj}$ equal to the $j$th largest column total of the matrix
		\begin{displaymath}
\bordermatrix{\text{}&s^{u}_{1\subdot} & s^{u}_{2\subdot} & \ldots & s^{u}_{k\subdot}\cr
\tilde{\mu}^{u}\mu^1 & \tilde{\mu}^{u}\mu^1 s^{u}_{1,\sigma^{u}_1(1)} & \tilde{\mu}^{u}\mu^1 s^{u}_{1,\sigma^{u}_1(2)}&\ldots& \tilde{\mu}^{u}\mu^1 s^{u}_{1,\sigma^{u}_1(k)}\cr
\tilde{\mu}^{u}\mu^2 & \tilde{\mu}^{u}\mu^2 s^{u}_{2,\sigma^{u}_2(1)} & \tilde{\mu}^{u}\mu^2 s^{u}_{2,\sigma^{u}_2(2)}&\ldots& \tilde{\mu}^{u}\mu^2 s^{u}_{2,\sigma^{u}_2(k)}\cr
\vdots & \vdots & \vdots & \ddots & \vdots\cr
\tilde{\mu}^{u}\mu^k & \tilde{\mu}^{u}\mu^k s^{u}_{k,\sigma^{u}_k(1)} & \tilde{\mu}^{u}\mu^k s^{u}_{k,\sigma^{u}_k(2)}&\ldots& \tilde{\mu}^{u}\mu^k s^{u}_{k,\sigma^{u}_k(k)}}
\end{displaymath}
i.e. $\tilde{\mu}^{uj}:=\left(\sum_{i=1}^k\tilde{\mu}^u\mu^is^{u}_{i,\sigma^u_i(m)},m=1,\ldots,k\right)^{\downarrow}_j,$ where $\mu^1,\ldots,\mu^k$ correspond to the mass fragmentation of the root of $\mu$.
\end{itemize}
\begin{defn}\label{defn:associated mass fragmentation} For a fragmentation tree $T\in\mathcal{T}$, we write $\mathbb{M}(T)$ to denote the associated mass fragmentation of $T$, i.e.\ the mass fragmentation of $1$ obtained by replacing each child of $T$ by its asymptotic frequency, if it exists.\end{defn}
\begin{thm}\label{thm:associated mass fragmentation chain}Let $\mathbb{T}:=(T_n,n\geq1)$ be a $\cp(\nu)$-ancestral branching Markov chain with transition measure $Q(\cdot,\cdot;\nu)$ on $\mathcal{T}^{(k)}$, with initial distribution $\Pi$ some infinitely exchangeable measure on $\mathcal{T}$. Let $\mu:=(\mu_n,n\geq1)$ be the Markov chain on $\mathcal{M}_1^{(k)}$ generated from the above procedure, then $\mathbb{M}(\mathbb{T})=_{\mathcal{L}}\mu$.  Moreover, the transition measure $\lambda(\cdot,\cdot;\nu)$ for $\mu$ is given by 
$$\lambda(\mu,\mu';\nu)=Q(T_{\mu},\mathbb{M}^{-1}(\mu');\nu)$$
where $T_{\mu}$ is any element of $\mathbb{M}^{-1}(\mu):=\{T\in\treesk:\mathbb{M}(T)=\mu\}$.\end{thm}
\begin{proof}
Fix $k\geq2$ and $\nu$ a probability measure on $\masspartitionk$.  For $\mathbb{T}\sim Q(\cdot,\cdot;\nu)$ we have that for every $n\geq1$ and $t\in T_n$, the set of children $\{t_1,\ldots,t_m\}$ of $t$ forms an exchangeable partition of $\{t\}\subset\mathbb{N}$ given $T_{n-1}$ and so possesses asymptotic frequency $\vert\vert t\vert\vert$ almost surely by Kingman's correspondence.

The alternative construction of the Markov chain $\mathbb{T}$ with transition measure $Q(\cdot,\cdot;\nu)$ constructed in section \ref{section:alternative construction} can also be constructed as follows.  Let $S:=\{s^u:u\in\mathcal{U}\}$ be the collection of mass partitions in the construction at the beginning of this section.  Given $S$, generate $B:=\{B^u:u\in\mathcal{U}\}\in\prod_{u\in\mathcal{U}}\left[\prod_{i=1}^k\partitionsk\right]$ by letting $B^u:=(B^u_1,\ldots,B^u_k)$ and $B^u_j\sim\varrho_{s^{u}_j}$ independently of all other $B^v_i$.  Constructed in this way, $\{B^u:u\in\mathcal{U}\}$ is a collection of i.i.d.\ $\varrho_{\nu}^{(k)}$ partitions whose asymptotic frequencies satisfy $\vert\vert B^u_j\vert\vert = s^u_j$ almost surely.  Furthermore, the unconditional distribution of each $B^u$ is $\varrho_{\nu}^{(k)}$.

Next, we let $\Sigma:=\{\sigma^u:u\in\mathcal{U}\}$ be a collection of i.i.d.\ $k$-tuples of i.i.d.\ uniform permutations of $[k]$ and generate transitions of $\mathbb{T}$ from the alternative construction of section \ref{section:alternative construction} based on $\Sigma$ and $\{B^u:u\in\mathcal{U}\}$ and generate a Markov chain $\mu$ on $\mathcal{M}_1$ based on $\Sigma$ and $S$.  Then we have the $\mathbb{T}$ is a Markov chain with transition measure $Q(\cdot,\cdot;\nu)$ on $\left(\treesk,\sigma\left(\bigcup_{n\geq1}\treesnk\right)\right)$ and, furthermore, by the above construction, we have that the associated mass fragmentation chain $\mathbb{M}(\mathbb{T}):=(\mathbb{M}(\mathbb{T}_n),n\geq1)$ is equal to $\mu$ almost surely.

By the three step construction of transitions on $\mathcal{M}_1$ at the beginning of this section, it is clear that $\mu$ is a Markov chain.  Hence, the function $\mathbb{M}(\mathbb{T})$ is a Markov chain and so the result of Burke and Rosenblatt \cite{BurkeRosenblatt1958} states that it is necessary that the transition measure of $\mathbb{M}(\mathbb{T})$ satisfies
$$Q\mathbb{M}^{-1}(m,m';\nu)=\int_{\mathbb{M}^{-1}(m')}Q(T_m,dt)$$
for all $T_m\in\mathbb{M}^{-1}(m):=\{T\in\mathcal{T}:\mathbb{M}(T)=m\}$.

Finally, since $\mathbb{M}(\mathbb{T})=\mu$ almost surely, we have that the transition measure $\lambda$ of $\mu$ on $\mathcal{M}_1$ satisfies $\lambda=\pi\mathbb{M}^{-1}$.
\end{proof}
\begin{cor}The associated mass fragmentation process $\mathbb{M}(\mathbb{T})$ exists almost surely.\end{cor}
\subsection{Equilibrium measure}
As in section \ref{section:equilibrium measure}, suppose $\nu$ is non-degenerate at $(1,0,\ldots,0)\in\masspartitionk$.  Theorem \ref{cor:existence inf exch stationary} states that a Markov chain $\mathbb{T}:=(T_n,n\geq1)$ governed by $Q(\cdot,\cdot;\nu)$ possesses a unique equilibrium measure $\rho(\cdot;\nu)$.  The following theorem follows immediately from this fact and from theorem \ref{thm:associated mass fragmentation chain}.
\begin{thm}\label{thm:mass fragmentation stationary measure}Let $\nu$ be a probability measure on $\masspartitionk$ such that $\nu((1,0,\ldots,0))<1$.  The mass fragmentation chain $\mu:=(\mu_n,n\geq1)$ on $\mathcal{M}_1$ governed by $Q\mathbb{M}^{-1}(\cdot,\cdot;\nu)$ possesses a unique stationary measure $\zeta(\cdot;\nu)$.  Moreover, for $\mu\in\mathcal{M}_1^{(k)}$, 
$$\zeta(\mu;\nu)=\rho(\mathbb{M}^{-1}(\mu);\nu)$$
where $\rho(\cdot;\nu)$ is the unique equilibrium measure of $Q(\cdot,\cdot,\nu)$ on $\treesk$ from corollary \ref{cor:existence inf exch stationary}.
\end{thm}
\begin{proof}
Let $\mu$ be a Markov chain on $\mathcal{M}_1$ with transition measure $\lambda(\cdot,\cdot;\nu)$ governed by the transition procedure at the beginning of section \ref{section:mass fragmentations}.  By theorem \ref{thm:associated mass fragmentation chain} we have that $\lambda\equiv Q\mathbb{M}^{-1}$ where $Q(\cdot,\cdot;\nu)$ is the transition measure of the $\cp(\nu)$-ancestral branching Markov chain on $\treesk$ with unique equilibrium measure $\rho(\cdot;\nu)$ from corollary \ref{cor:existence inf exch stationary}.

 Furthermore, it is shown in theorem \ref{thm:associated mass fragmentation chain} that $\mu$ is equal in distribution to the associated mass fragmentation chain of a Markov chain on $\treesk$ governed by $\pi(\cdot,\cdot;\nu)$.  Hence, we have 
$$\rho(\tau';\nu)=\int_{\treesk}Q(\tau,\tau';\nu)\rho(d\tau)$$
and for $\mu'\in\mathcal{M}_1$
\begin{eqnarray*}
\rho\mathbb{M}^{-1}(\mu';\nu)&=&\rho[\mathbb{M}^{-1}(\mu);\nu]\\
&=&\int_{\mathbb{M}^{-1}(\mu)}\int_{\treesk}Q(\tau,dt;\nu)\rho(d\tau;\nu)\\
&=&\int_{\treesk}Q(\tau,\mathbb{M}^{-1}(\mu');\nu)\rho(d\tau;\nu)\\
&=&\int_{\mathcal{M}_1}Q\mathbb{M}^{-1}(\mu,\mu';\nu)\rho\mathbb{M}^{-1}(d\mu)\\
&=&\int_{\mathcal{M}_1}\lambda(\mu,\mu';\nu)\rho\mathbb{M}^{-1}(d\mu)\end{eqnarray*}
which shows that $\zeta:=\rho\mathbb{M}^{-1}$ is stationary for $\lambda$.
\end{proof}
\subsection{Poissonian construction}\label{section:poissonian construction mass fragmentation}
Just as the $\cp(\nu)$-ancestral branching process on $\treesk$ admits a Poissonian construction, which we showed in section \ref{section:poissonian construction}, so does its associated mass fragmentation-valued process, which we now show.

Let $\nu$ be a probability measure on $\masspartitionk$.  Let $S=\{(t,s^u):u\in\mathcal{U}\}\subset\mathbb{R}^+\times\prod_{u\in\mathcal{U}}\left[\prod_{i=1}^k\masspartitionk\right]$ be a Poisson point process with intensity $dt\otimes\lambda\bigotimes_{u\in\mathcal{U}}\nu^{(k)}$ for some $\lambda>0$ where $\nu^{(k)}:=\nu\otimes\cdots\otimes\nu$ is the $k$-fold product measure on $\prod_{i=1}^k\masspartitionk$ and $s^u:=(s^{u}_1,\ldots,s^{u}_k)\in\prod_{i=1}^k\masspartitionk$ for each $u\in\mathcal{U}$.

Construct a Markov process $\mu:=(\mu(t),t\geq0)$ in continuous-time on $\mathcal{M}_1$ as follows.  Let $\mu_0$ be a mass fragmentation drawn from some distribution on $\mathcal{M}_1$.  Put $\mu(0)=\mu_0$ and
\begin{itemize}
	\item if $t$ is not an atom time for $S$, $\mu(t)=\mu(t-)$;
	\item if $t$ is an atom time for $S$, generate $\Sigma_t:=\{\sigma^u:u\in\mathcal{U}\}$ where $\sigma^v$ and $\sigma^w$ are independent for all $v\neq w$ and $\sigma^u:=(\sigma^u_1,\ldots,\sigma^u_k)$ is an i.i.d.\ sequence of uniform permutations of $[k]$ for each $u\in\mathcal{U}$.  Given $(t,s^u)\in S$, $\sigma^u$ and $\mu(t-)=\{\mu^u:u\in\mathcal{U}\}$, put $\mu(t)=\{\tilde{\mu}^u:u\in\mathcal{U}\}$ where
\begin{itemize}
	\item[1)] $\tilde{\mu}^{\emptyset}=1$ and
	\item[2)] given $\tilde{\mu}^u$, put $\tilde{\mu}^{uj}$ equal to the $j$th largest column total of the matrix
	\begin{displaymath}
\bordermatrix{\text{}&s^{u}_{1\subdot} & s^{u}_{2\subdot} & \ldots & s^{r_i}_{k\subdot}\cr
\tilde{\mu}^{u}\mu^1 & \tilde{\mu}^{u}\mu^1 s^{u}_{1,\sigma^{u}_1(1)} & \tilde{\mu}^{u}\mu^1 s^{u}_{1,\sigma^{u}_1(2)}&\ldots& \tilde{\mu}^{u}\mu^1 s^{u}_{1,\sigma^{u}_1(k)}\cr
\tilde{\mu}^{u}\mu^2 & \tilde{\mu}^{u}\mu^2 s^{u}_{2,\sigma^{u}_2(1)} & \tilde{\mu}^{u}\mu^2 s^{u}_{2,\sigma^{u}_2(2)}&\ldots& \tilde{\mu}^{u}\mu^2 s^{u}_{2,\sigma^{u}_2(k)}\cr
\vdots & \vdots & \vdots & \ddots & \vdots\cr
\tilde{\mu}^{u}\mu^k & \tilde{\mu}^{u}\mu^k s^{u}_{k,\sigma^{u}_k(1)} & \tilde{\mu}^{u}\mu^k s^{u}_{k,\sigma^{u}_k(2)}&\ldots& \tilde{\mu}^{u}\mu^k s^{u}_{k,\sigma^{u}_k(k)}}
\end{displaymath}
i.e. $\tilde{\mu}^{uj}:=\left(\sum_{i=1}^k\tilde{\mu}^u\mu^is^{u}_{i,\sigma^u_i(m)},m=1,\ldots,k\right)^{\downarrow}_j.$
\end{itemize}
\end{itemize}
\begin{thm}\label{thm:associated mass fragmentation}Let $\mathbb{T}:=(T(t),t\geq0)$ be a $\cp(\nu)$-ancestral branching Markov process from section \ref{section:continuous-time} and let $X:=(X(t),t\geq0)$ be the Markov process on $\mathcal{M}_1$ generated from the above Poisson point process, then $\mathbb{M}(T)=_{\mathcal{L}}X$.\end{thm}
\begin{proof}
Let $k\in\mathbb{N}$ and $\nu$ be a measure on $\masspartitionk$.

Let $S=\{(t,s^u):u\in\mathcal{U}\}\subset\mathbb{R}^+\times\prod_{u\in\mathcal{U}}\left[\prod_{i=1}^k\masspartitionk\right]$ be a Poisson point process with intensity $dt\otimes\lambda\bigotimes_{u\in\mathcal{U}}\nu^{(k)}$ for some $\lambda>0$ as shown above and let $X:=(X(t),t\geq0)$ be the process on $\mathcal{M}_1$ constructed above.  Given $S$, generate $P:=\{(t,B^u):u\in\mathcal{U}\}\subset\mathbb{R}^+\times\prod_{u\in\mathcal{U}}\left[\prod_{i=1}^k\mathcal{P}^{(k)}\right]$ where for each $(t,s^u:u\in\mathcal{U})\in S$ we let $B^u:=(B^u_1,\ldots,B^u_k)\in\prod_{i=1}^k\mathcal{P}^{(k)}$ be a $k$-tuple of partitions such that $B^u_i\sim\varrho_{s^{u}_i}$ for each $i=1,\ldots,k$ and all components are independent.  Thus, we have that $P$ is a Poisson point process on $\mathbb{R}^+\times\prod_{u\in\mathcal{U}}\left[\prod_{i=1}^k\mathcal{P}^{(k)}\right]$ with intensity measure $dt\otimes\lambda\bigotimes_{u\in\mathcal{U}}\varrho_{\nu}^{(k)}$.  Given $P$ and $S$, generate $\Sigma:=\{\sigma^u:u\in\mathcal{U}\}$ independently of $P$ and $S$ such that $\sigma^v$ and $\sigma^w$ are independent for all $v\neq w$ and each $\sigma^u=(\sigma^u_1,\ldots,\sigma^u_k)$ is an i.i.d.\ collection of uniform permutations of $[k]$.

Let $\mathbb{T}:=(T(t),t\geq0)$ be the process on $\treesk$ constructed from $\Sigma$ and $P$, as shown in section \ref{section:poissonian construction}, so that $\mathbb{T}$ is a $\cp(\nu)$-ancestral branching Markov process. Likewise, let $X:=(X(t),t\geq0)$ be the process on $\mathcal{M}_1$ constructed from $\Sigma$ and $S$ shown above.

Now for all $t\geq0$, let $T(t-)=\tau$.  Then $T(t)=\tilde{\tau}$ where
$$\tilde{\tau}^{uj}=\tilde{\tau}^u\bigcap\left(\bigcup_{i=1}^k(\tau^i\cap B^{u}_{i,\sigma^u_i(j)})\right)$$ for each $u\in\mathcal{U}$ and $j=1,\ldots,k$ which has asymptotic frequency
$$\vert\vert\tilde{\tau}^u\vert\vert\sum_{i=1}^k\vert\vert\tau^i\vert\vert\vert\vert B^{u}_{i,\sigma^u_i(j)}\vert\vert=\tilde{\mu}^u\sum_{i=1}^k\mu^i s^{u}_{i,\sigma^u_i(j)}\quad\mbox{a.s.}$$
Hence we have that $\mu=\mathbb{M}(\mathbb{T})$ a.s.\ in this construction and so $\mu=_{\mathcal{L}}\mathbb{M}(\mathbb{T})$.
\end{proof}
\begin{cor}\label{cor:existence mass fragmentation}The process $\mathbb{M}(\mathbb{T}):=(\mathbb{M}(T(t)),t\geq0)$ exists almost surely.\end{cor}
%\begin{cor}\label{cor:mass fragmentation tps}The transition probabilities of $\Lambda(T)$ are given by $(\lambda(m,\cdot;\nu), m\in\mathcal{M}_1)$ where
%\begin{equation}\lambda(m,m';\nu)=\pi(\tau_m,\Lambda^{-1}(m');\nu)=\int_{\Lambda^{-1}(m')}\pi(\tau_m,d\mu;\nu)\label{eq:mass fragmentation tps}
%\end{equation}
%where $\pi(\cdot,\cdot;\nu)$ is the transition probability on $\treesk$ in theorem \ref{thm:existence chain} and $\tau_m$ is any element of $\Lambda^{-1}(m)$.\end{cor}
%\begin{proof}
%In theorem \ref{thm:existence chain} the transition measure $\pi(\cdot,\cdot;\nu)$ is shown to exist on the measurable space $\left(\treesk,\sigma\left(\bigcup_{n\geq1}\treesnk\right)\right)$.  
%
%First, we show that $\Lambda:\mathcal{T}\rightarrow\mathcal{M}_1$ defined above is a measurable function between $\sigma\left(\bigcup_{n\geq1}\treesnk\right)$ and $\Lambda\left(\sigma\left(\bigcup_{n\geq1}\treesnk\right)\right)$.
%\end{proof}
%\subsection{Equilibrium measure}
%In section \ref{section:equilibrium measure} it is shown that the tree-valued Markov process $T:=(T(t),t\geq0)$ from theorem \ref{thm:existence chain} possesses a unique infinitely exchangeable equilibrium measure $\rho(\cdot;\nu)$ on $\treesk$.  In section \ref{section:poissonian construction mass fragmentation} it was shown that for a $\varrho_{\nu}$-Markov fragmentation process $T$, its associated mass fragmentation $\Lambda(T)$ exists almost surely and admits a construction by a Poisson point process.
\section{Weighted trees}\label{section:weighted trees}
A {\em weighted tree} is a fragmentation tree with edge lengths.  We write $\bar{\mathcal{T}}:=\mathcal{T}\times\left(\mathbb{R}^+\right)^{\mathcal{U}}$ to denote the space of weighted trees; i.e.\ each $\bar{T}\in\bar{\mathcal{T}}$ is a pair $(T,\{t_b:b\in T\})$ consisting of a fragmentation tree $T$ and a set of edge lengths corresponding to each edge of the tree with the convention that $t_b\equiv0$ if $b\notin T$.  We prefer the term {\em weighted tree} to the alternative {\em fragmentation process} which is generally thought of as a non-increasing sequence of random partitions of $\mathbb{N}$, $B:=(B(t),t\geq0)$, indexed by $t\in\mathbb{R}^+$, i.e.\ $B(t)\leq B(s)$ for all $t\geq s$.  By referring to these objects as weighted trees, we hope to emphasize $\bar{T}\in\bar{\mathcal{T}}$ as an object, rather than a process.  In this way, our construction of a Markov process on $\bar{\mathcal{T}}^{(k)}$ is naturally interpreted as a random walk on this space of objects with only one temporal component, that being how our process on $\bar{\mathcal{T}}^{(k)}$ evolves in time.

In section \ref{section:cut-and-paste ancestral branching} we introduce the $\cp(\nu)$ family of AB transition probabilities $Q_n(T,\cdot;\nu)$ for each $k\geq2$, $T\in\treesnk$ and $\nu$ a probability measure on $\masspartitionk$.  The results of section \ref{section:infinitely exchangeable processes} and \ref{section:equilibrium measure} establish the existence of a transition measure $Q(T,\cdot;\nu)$ on $\treesk$ with infinitely exchangeable stationary measure $\rho(\cdot;\nu)$.

We now construct a transition probability on $\bar{\mathcal{T}}^{(k)}$.  Let $\bar{T}=(T,\{t_b:b\in T\})\in\bar{\mathcal{T}}_n^{(k)}$ and generate $\bar{T}'=(T',\{t'_b:b\in T'\})\in\bar{\mathcal{T}}_n^{(k)}$ by the following two-step procedure.

\paragraph{Ancestral Branching with edge lengths Algorithm}

\begin{enumerate}
	\item[(i)] Generate $T'$ from $Q_n(T,\cdot;\nu)$;
	\item[(ii)] given $T'$, generate each $t'_b$ from an exponential distribution with rate parameter $\theta q_b(\Pi_{T_{|b}},{\bf1}_b;\nu)$ (i.e.\ mean $1/\theta q_b(\Pi_{T_{|b}},{\bf1}_b;\nu)$) independently for each $b\in T'$, for some $\theta>0$.
\end{enumerate}
This procedure yields a transition density on $\bar{\mathcal{T}}_n^{(k)}$ given by 
\begin{equation}\bar{Q}_n(\bar{T},\bar{T}';\nu)=\prod_{b\in T'}\theta p_b(\Pi_{T_{|b}},\Pi_{T'_{|b}};\nu)e^{-\theta t'_b q_b(\Pi_{T_{|b}},{\bf1}_b;\nu)}dt'_b.\label{eq:transition density}\end{equation}
The purpose of choosing each waiting time $t'_b$ to be an exponential random variable with parameter $\theta q_b(\Pi_{T_{|b}},{\bf1}_b;\nu)$ is to ensure the consistency of the process under restriction.

Consider $\bar{T}=(T,\{t_b:b\in T\})$ and $\bar{T}^*=(T^*,\{t^*_b:b\in T^*\})$ such that $T^*\in D^{-1}_{n,n+1}(T)$.  Then $T^*$ has a vertex $A\cup\{n+1\}$ with children $\{n+1\}$ and $A\in T$.  This is the branch of $T$ on which the leaf $\{n+1\}$ is attached.  Denote this vertex by $A^*\in T^*$ and require that $t^*_b=t_b$ for $b\notin\{A^*,A\}$ and $t^*_{A^*}+t^*_A=t_A$.  We denote by $\bar{D}_{n,n+1}^{-1}(\bar{T})$ the set of $\bar{T}^*$ satisfying these conditions.

Consistency requires that for a tree $\bar{T}''\sim\bar{Q}_{n+1}(\bar{T}^*,\cdot;\nu)$, the restriction $\bar{T}':=\bar{T}''_{|[n]}$ is distributed as $\bar{Q}_n(\bar{T}^*_{|[n]},\cdot;\nu)$.

\begin{prop}\label{prop:consistent process tps}Let $\nu$ be a probability measure on $\masspartitionk$, $n\geq1$, $\bar{T}^*\in\bar{\mathcal{T}}_{n+1}^{(k)}$ and $\bar{T}''\sim\bar{Q}_{n+1}(\bar{T}^*,\cdot;\nu)$.  Then the restriction $\bar{T}':=\bar{T}''_{|[n]}$ is distributed as $\bar{Q}_n(T^*_{|[n]},\cdot;\nu)$.\end{prop}
\begin{proof}
Let $\bar{T}^*=(T^*,\{t^*_b:b\in T^*\})\in\bar{\mathcal{T}}_{n+1}^{(k)}$ and $\bar{T}''=(T'',\{t''_b:b\in T''\})\in\bar{\mathcal{T}}_{n+1}^{(k)}$.  By construction of $\bar{Q}_{n}(\cdot,\cdot;\nu)$ on $\bar{\mathcal{T}}_{n}^{(k)}$ for each $n\geq1$, we have that $T''_{|[n]}\sim Q_n(T^*_{|[n]},\cdot;\nu)$ and the induced process on boolean trees is consistent.

Let $t''_{n+1}$ denote the length of the root edge of $\bar{T}''$ and consider the length of the root edge of the restriction $\bar{T}''_{|[n]}$, denoted $t'_n$.  If $\Pi_{T''}\neq{\bf e}_{n+1}$, then $t'_n=t''_{n+1}$.  Otherwise, $t'_n=t''_{n+1}+t''_n$.  Hence, $t'_n\sim\tau+\tau'\mathbb{I}_{A}$ where $\tau$ and $\tau'$ are, respectively, independent exponential random variables with parameters $\theta q_{n+1}(\Pi_{T^*},{\bf1}_{n+1};\nu)$ and $\theta q_n(\Pi_{T^*_{|[n]}},{\bf1}_n;\nu)$ for some $\theta>0$ and $A:=\{\Pi_{T''}={\bf e}_{n+1}\}$, the event that the children of the root $[n+1]$ in $T''$ are $[n]$ and $\{n+1\}$, is independent of $\tau$ and $\tau'$.

For notational convenience, we drop the dependence on $\nu$ and write $q_b(\cdot,\cdot)\equiv q_b(\cdot,\cdot;\nu)$ for any $b\subset\mathbb{N}$, likewise for $p_b(\cdot,\cdot;\nu)$, where $q_n$ and $p_n$ are defined in section \ref{section:cut-and-paste algorithm}.

An exponential random variable with rate parameter $\lambda>0$ has moment generating function $\mathcal{E}_{\lambda}(t):=\lambda/(\lambda-t)$.  The moment generating function of $t'_n$ is
\small
\begin{eqnarray}
\lefteqn{\mathbb{E}e^{t(\tau+\tau'\mathbb{I}_{A})}=}\notag\\
&=&\mathbb{E}e^{t\tau}\mathbb{E}e^{t\tau'\mathbb{I}_{A}}\label{prop:consistent edge length:line1}\\
&=&\frac{\theta q_{n+1}(\Pi_{T^*},{\bf1}_{n+1})}{\theta q_{n+1}(\Pi_{T^*},{\bf1}_{n+1})-t}\left[\mathbb{E}\left(e^{t\tau'\mathbb{I}_{A}}|A\right)\mathbb{P}(A)+\mathbb{E}\left(e^{t\tau'\mathbb{I}_{A}}|A^{c}\right)\mathbb{P}(A^{c})\right]\label{prop:consistent edge length:line2}\\
&=&\frac{\theta q_{n+1}(\Pi_{T^*},{\bf1}_{n+1})}{\theta q_{n+1}(\Pi_{T^*},{\bf1}_{n+1})-t}\left[\frac{p_{n+1}(\Pi_{T^*},{\bf e}_{n+1})}{q_{n+1}(\Pi_{T^*},{\bf1}_{n+1})}\frac{\theta q_n(\Pi_{T^*_{|[n]}},{\bf1}_n)}{\theta q_n(\Pi_{T^*_{|[n]}},{\bf1}_n)-t}+1-\frac{p_{n+1}(\Pi_{T^*},{\bf e}_{n+1})}{q_{n+1}(\Pi_{T^*},{\bf1}_{n+1})}\right]\label{prop:consistent edge length:line3}\\
&=&\frac{\theta q_{n+1}(\Pi_{T^*},{\bf1}_{n+1})}{\theta q_{n+1}(\Pi_{T^*},{\bf1}_{n+1})-t}\left[\frac{p_{n+1}(\Pi_{T^*},{\bf e}_{n+1})\theta q_n(\Pi_{T^*_{|[n]}},{\bf1}_n)+q_{n+1}(\Pi_{T^*},{\bf1}_{n+1})(\theta q_n(\Pi_{T^*_{|[n]}},{\bf1}_n)-t)}{q_{n+1}(\Pi_{T^*},{\bf1}_{n+1})(\theta q_n(\Pi_{T^*_{|[n]}},{\bf1}_n)-t)}-\right.\notag\\
&&\quad\left.-\frac{p_{n+1}(\Pi_{T^*},{\bf e}_{n+1})(\theta q_n(\Pi_{T^*_{|[n]}},{\bf1}_n)-t)}{q_{n+1}(\Pi_{T^*},{\bf1}_{n+1})(\theta q_n(\Pi_{T^*_{|[n]}},{\bf1}_n)-t)}\right]\label{prop:consistent edge length:line4}\\
&=&\frac{\theta q_{n+1}(\Pi_{T^*},{\bf1}_{n+1})}{\theta q_{n+1}(\Pi_{T^*},{\bf1}_{n+1})-t}\left[\frac{q_{n+1}(\Pi_{T^*},{\bf1}_{n+1})\theta q_n(\Pi_{T^*_{|[n]}},{\bf1}_n)-tq_{n+1}(\Pi_{T^*},{\bf1}_{n+1})}{q_{n+1}(\Pi_{T^*},{\bf1}_{n+1})(\theta q_n(\Pi_{T^*_{|[n]}},{\bf1}_n)-t)}\right.\notag\\
&&\quad\left.+\frac{tp_{n+1}(\Pi_{T^*},{\bf e}_{n+1})}{q_{n+1}(\Pi_{T^*},{\bf1}_{n+1})(\theta q_n(\Pi_{T^*_{|[n]}},{\bf1}_n)-t)}\right]\label{prop:consistent edge length:line5}\\
&=&\frac{\theta q_{n+1}(\Pi_{T^*},{\bf1}_{n+1})}{\theta q_{n+1}(\Pi_{T^*},{\bf1}_{n+1})-t}\left[\frac{q_{n+1}(\Pi_{T^*},{\bf1}_{n+1})\theta q_n(\Pi_{T^*_{|[n]}},{\bf1}_n)-tq_n(\Pi_{T^*_{|[n]}},{\bf1}_n)}{q_{n+1}(\Pi_{T^*},{\bf1}_{n+1})(\theta q_n(\Pi_{T^*_{|[n]}},{\bf1}_n)-t)}\right]\label{prop:consistent edge length:line6}\\
&=&\frac{\theta q_n(\Pi_{T^*_{|[n]}},{\bf1}_n)}{\theta q_n(\Pi_{T^*_{|[n]}},{\bf1}_n)-t}\label{prop:consistent edge length:line7}
\end{eqnarray}
\normalsize
the moment generating function of $\tau'$.

Line \eqref{prop:consistent edge length:line1} follows by independence of $\tau, \tau'$ and $A$; \eqref{prop:consistent edge length:line2} uses the tower property of conditional expections; \eqref{prop:consistent edge length:line3} substitutes explicit expressions for the expression in \eqref{prop:consistent edge length:line2}; \eqref{prop:consistent edge length:line5} is obtained from \eqref{prop:consistent edge length:line4} by canceling terms in the numerator; \eqref{prop:consistent edge length:line6} follows \eqref{prop:consistent edge length:line5} by fact that $q_n(\Pi_{T^*_{|[n]}},{\bf1}_n)=q_{n+1}(\Pi_{T^*},{\bf1}_{n+1})-p_{n+1}(\Pi_{T^*},{\bf e}_{n+1})$ by consistency of \eqref{eq:fidi tps}; finally, \eqref{prop:consistent edge length:line7} is obtained by simplifying the expression \eqref{prop:consistent edge length:line6}.

By the branching property of $\bar{Q}_{n}(\cdot,\cdot;\nu)$ we have that the restriction $\bar{T}''_{|[n]}$ is distributed as $\bar{Q}_n(\bar{T}^*_{|[n]},\cdot;\nu)$.
\end{proof}
Finite exchangeability is immediate by inspecting the form of \eqref{eq:transition density}.  The existence of a transition density on $\bar{\mathcal{T}}^{(k)}$ is once again immediate by Kolmogorov's theorem.
\begin{thm}There exists a transition density $\bar{Q}(\cdot,\cdot;\nu)$ on $\bar{T}^{(k)}$ whose finite-dimensional restrictions are given by \eqref{eq:transition density}.\end{thm}
%Embedding the above process on $\bar{\mathcal{T}}^{(k)}$ in continuous-time is trivially done by adding an exponentially distributed wait time with mean $1/\lambda>0$ between each transition.  With probability one, $\bar{T}$ jumps to a state other than $\bar{T}$ and so the collection of restrictions $\bar{T}_{|[n]}:=(\bar{T}_{|[n]}(t),t\geq0)$ are consistent and define a process on $\bar{T}^{(k)}$ with c\`adl\`ag paths.
The above process on weighted trees for the $\cp(\nu)$-ancestral branching process on $\treesk$ is straightforward to construct, mainly due to the restriction to trees with a bounded number of children, i.e.\ each parent can have no more than $k\geq1$ children.  For this reason, we do not run into issues in our specification related to the accumulation of an infinite number of partition events.  On one hand, this restriction makes the existence of the above process uninteresting probabilistically as we restrict our attention to only a finite number of events.  On the other hand, this provides an explicit, easily implemented, procedure for generating a random sequence of, for example, binary trees, which could be of interest in certain applications.
\section{Discussion}
Here we have shown an explicit construction of a Markov process on $\mathcal{T}$ and $\mathcal{P}$ via, respectively, the ancestral branching and cut-and-paste algorithms, and under what conditions the AB algorithm characterizes the transition probabilities of an infinitely exchangeable tree-valued process.  There is potentially a wealth of interesting work that can be done by exploring this family of processes in more detail.  We provide some details on the ancestral branching process associated with the transition probabilities of the cut-and-paste process with parameter $\nu$, where $\nu$ is a measure on the ranked-$k$ simplex.  In this case, the associated tree-valued process is restricted to $\treesk$.  A process based on a more general form of the cut-and-paste algorithm, which is not restricted to trees with a bounded number of children, could be interesting to study.  However, the case that we study is also interesting, in particular in the case where $k=2$ and we have an infinitely exchangeable process on the space of binary trees.

For the parametric subfamily of the $\cp(\nu)$-process with $\nu=\PD(-\alpha/k,\alpha)$, the finite-dimensional transition probabilities on $\mathcal{T}_n^{(2)}$ for $n\geq1$, $\alpha>0$ and $t,t'\in\mathcal{T}_n^{(2)}$ is given by
$$Q_n(t,t';\alpha)=\prod_{b\in\Pi_{t'}:\#b\geq2}\frac{2\per_{\alpha/2}(B\wedge B')}{\per_{\alpha}B-2\per_{\alpha/2}B},$$
where $\per_{\alpha}B$ represents the $\alpha$-permanent of $B$, regarded as a 0-1 valued boolean matrix.

Implications of this subfamily to inferring unknown phylogenetic trees and also to hidden Markov modeling in a genetic framework are potentially viable applications of this process.  Furthermore, the $\cp(\alpha,k)$ subfamily is known to be reversible with respect to the Pitman-Ewens family of distributions with parameter $(-\alpha,k\alpha)$, yet it is not immediately clear whether this has implications for the equilibrium measure of the associated $\cp(\alpha,k)$-ancestral branching process.  Connections between these equilibrium measures, and their relationship to Aldous's continuum random tree \cite{AldousCRTI} are of interest in this space.
\bibliography{crane-refs}

\begin{thebibliography}{10}

\bibitem{BertoinPIMS}
Bertoin, J. (2010). \it Exchangeable Coalescents\rm, Lecture notes for PIMS
  Summer School in Probability 2010, University of Washington and Microsoft
  Research.

\bibitem{AldousCRTI}
D.~Aldous.
\newblock The continuum random tree. {I}.
\newblock {\em Ann. Probab.}, 19(1):1--28, 1991.

\bibitem{AldousPitmanTree}
D.~Aldous and J.~Pitman.
\newblock Tree-valued {M}arkov chains derived from {G}alton-{W}atson processes.
\newblock {\em Ann. Inst. H. Poincar\'e Probab. Statist.}, 34(5):637--686,
  1998.

\bibitem{Berestycki2004}
J.~Berestycki.
\newblock Exchangeable fragmentation-coalescence processes and their
  equilibrium measures.
\newblock {\em Electron. J. Probab.}, 9:no. 25, 770--824 (electronic), 2004.

\bibitem{Bertoin2001a}
J.~Bertoin.
\newblock Homogeneous fragmentation processes.
\newblock {\em Probab. Theory Related Fields}, 121(3):301--318, 2001.

\bibitem{Bertoin2002}
J.~Bertoin.
\newblock Self-similar fragmentations.
\newblock {\em Ann. Inst. H. Poincar\'e Probab. Statist.}, 38(3):319--340,
  2002.

\bibitem{Bertoin2003a}
J.~Bertoin.
\newblock The asymptotic behavior of fragmentation processes.
\newblock {\em J. Eur. Math. Soc. (JEMS)}, 5(4):395--416, 2003.

\bibitem{Bertoin2006}
J.~Bertoin.
\newblock {\em Random fragmentation and coagulation processes}, volume 102 of
  {\em Cambridge Studies in Advanced Mathematics}.
\newblock Cambridge University Press, Cambridge, 2006.

\bibitem{Billingsley1995}
P.~Billingsley.
\newblock {\em Probability and measure}.
\newblock Wiley Series in Probability and Mathematical Statistics. John Wiley
  \& Sons Inc., New York, third edition, 1995.
\newblock A Wiley-Interscience Publication.

\bibitem{BurkeRosenblatt1958}
C.~J. Burke and M.~Rosenblatt.
\newblock A {M}arkovian function of a {M}arkov chain.
\newblock {\em Ann. Math. Statist.}, 29:1112--1122, 1958.

\bibitem{Crane2011a}
H.~Crane.
\newblock A consistent markov partition process generated from the paintbox
  process.
\newblock {\em J. Appl. Probab.}, 43(3):778--791, 2011.

\bibitem{DurrettGranovskyGueron1999}
R.~Durrett, B.~L. Granovsky, and S.~Gueron.
\newblock The equilibrium behavior of reversible coagulation-fragmentation
  processes.
\newblock {\em J. Theoret. Probab.}, 12(2):447--474, 1999.

\bibitem{EvansPitman1998}
S.~N. Evans and J.~Pitman.
\newblock Construction of {M}arkovian coalescents.
\newblock {\em Ann. Inst. H. Poincar\'e Probab. Statist.}, 34(3):339--383,
  1998.

\bibitem{EvansWinter2006}
S.~N. Evans and A.~Winter.
\newblock Subtree prune and regraft: a reversible real tree-valued {M}arkov
  process.
\newblock {\em Ann. Probab.}, 34(3):918--961, 2006.

\bibitem{Ewens1972}
W.~J. Ewens.
\newblock The sampling theory of selectively neutral alleles.
\newblock {\em Theoret. Population Biology}, 3:87--112; erratum, ibid. 3
  (1972), 240; erratum, ibid. 3 (1972), 376, 1972.

\bibitem{Felsenstein2004}
J.~Felsenstein.
\newblock {\em Inferring Phylogenies}.
\newblock Sinauer Associates, Inc., Sunderland, MA, 2004.

\bibitem{Kingman1978}
J.~F.~C. Kingman.
\newblock Random partitions in population genetics.
\newblock {\em Proc. Roy. Soc. London Ser. A}, 361(1704):1--20, 1978.

\bibitem{Kingman1978a}
J.~F.~C. Kingman.
\newblock Random partitions in population genetics.
\newblock {\em Proc. Roy. Soc. London Ser. A}, 361(1704):1--20, 1978.

\bibitem{Kingman1980}
J.~F.~C. Kingman.
\newblock {\em Mathematics of genetic diversity}, volume~34 of {\em CBMS-NSF
  Regional Conference Series in Applied Mathematics}.
\newblock Society for Industrial and Applied Mathematics (SIAM), Philadelphia,
  Pa., 1980.

\bibitem{Kingman1982}
J.~F.~C. Kingman.
\newblock The coalescent.
\newblock {\em Stochastic Process. Appl.}, 13(3):235--248, 1982.

\bibitem{McCullaghPitmanWinkel2008}
P.~McCullagh, J.~Pitman, and M.~Winkel.
\newblock Gibbs fragmentation trees.
\newblock {\em Bernoulli}, 14(4):988--1002, 2008.

\bibitem{Pitman2005}
J.~Pitman.
\newblock {\em Combinatorial stochastic processes}, volume 1875 of {\em Lecture
  Notes in Mathematics}.
\newblock Springer-Verlag, Berlin, 2006.
\newblock Lectures from the 32nd Summer School on Probability Theory held in
  Saint-Flour, July 7--24, 2002, With a foreword by Jean Picard.

\end{thebibliography}
\bibliographystyle{abbrv}

\end{document}